\documentclass[10pt]{amsart}
\usepackage{macrosS20-1} 
\usepackage[margin=0.9in]{geometry} 
\usepackage{appendix}
\usepackage{ mathrsfs }
\usepackage{xcolor}

\renewcommand{\P}{\mathbb P}

\newcommand{\m}{\mathcal}

\newcommand{\cal}{\mathcal}
\newcommand{\cL}{{\mathcal L}}
\newcommand{\cC}{{\mathcal C}}

\newcommand{\sP}{{\mathscr P}}

\newcommand{\dd}{\hspace{0.7pt}{\rm d}}

\DeclareMathOperator{\tr}{tr}

\newtheorem{assum}[thm]{Assumption}
\numberwithin{equation}{section}

\title[Well-Posedness of Displacement Monotone Degenerate MFG Master equations]{Global Well-Posedness of Displacement Monotone Degenerate Mean Field Games Master Equations}

\author[M. Bansil]{Mohit Bansil}
\address{Department of Mathematics, University of California, Los Angeles, CA 90095-1555, USA}
\email{mbansil@math.ucla.edu} 

\author[A.R. M\'esz\'aros]{Alp\'ar R. M\'esz\'aros}  
\date{\today}
\address{Department of Mathematical Sciences, University of Durham, Durham DH1 3LE, England}
\email{alpar.r.meszaros@durham.ac.uk} 

\author[C. Mou]{Chenchen Mou}
\address{Department of Mathematics, City University of Hong Kong, Hong Kong SAR, China}
\email{chencmou@cityu.edu.hk}

\begin{document}

\maketitle
\begin{abstract}
In this paper we construct global in time classical solutions to mean field games master equations in the lack of idiosyncratic noise in the individual agents' dynamics. These include both deterministic models and dynamics driven solely by a Brownian common noise. We consider a general class of non-separable Hamiltonians and final data functions that are supposed to be displacement monotone. Our main results unify and generalize in particular some of the well-posedness results on displacement monotone master equations obtained recently by Gangbo--M\'esz\'aros and Gangbo--M\'esz\'aros--Mou--Zhang. 
\end{abstract}

\section{Introduction}

\indent Master equations associated to mean field games (MFG) have been introduced by P.-L. Lions in his lectures \cite{Lions} at Coll\`ege de France. These are PDEs of hyperbolic type, whose solutions depend both on the state of individual agents (typically a variable in a finite dimensional Euclidean space) and on the agents' distribution (typically a Borel probability measure supported over the state space of the agents). Beside their independent interest, one of the main motivations for studying these equations lies in the fact that their classical solutions can be used to provide quantitative rates of convergence for the closed loop Nash equilibria of stochastic differential games, when the number of agents tends to infinity (cf. \cite{CarDelLasLio}). Although such rigorous convergence results were obtained in the presence of non-degenerate idiosyncratic noise only, we believe that classical solutions to the master equation when the idiosyncratic noise is degenerate (the ones that we provide in this paper), could potentially shed light on the convergence problem in more singular scenarios as well. To the best of our knowledge, the well-posedness of Nash system associated to the $N$-player game is open in the case of degenerate idiosyncratic noise, however, the well-posedness of master equation that we obtain in our manuscript could potentially be useful for the future proof of the convergence result. They can serve also as important tools in showing large deviation principles, concentration of measure and central limits theorems for these games (see \cite{DelLacRam19,DelLacRam20}).

\indent Because of the infinite dimensional character of these equations, their well-posedness provides great mathematical challenges and so their investigation has gained considerable attention in the community in the past decade. Classical solutions to the master equation are known to exist under certain assumptions on the data, which are responsible for the uniqueness of the MFG Nash equilibria of the underlying game. These assumptions can be roughly grouped into two categories: (i) the data satisfy some sort of {\it smallness condition} (related to the time horizon, to the Hamiltonian, to a specific subclass of probability measures, etc.) or (ii) the data fulfill suitable {\it monotonicity conditions}.

\indent In the case (i), besides the smallness assumption typically there is no need to impose additional structural assumptions on the data (such as separability or convexity of the underlying Hamiltonian or final datum or the presence of a non-degenerate idiosyncratic noise) governing the game (see for instance \cite{Carmona2018,GanSwi, May, AmbMes}). The question regarding the global well-posedness of master equations (in the class of classical solutions) is way more subtle and this is understood in suitably defined monotone regimes (cf. case (ii)). In the literature to date, these are essentially classified into the following groups: the so-called {\it Lasry--Lions monotonicity} and {\it displacement monotonicity} conditions. Historically, the Lasry--Lions monotonicity condition was used first for the global well-posedness of the master equation (see \cite{ChaCriDel22, CarDelLasLio,Carmona2018}). When dealing with classical solutions, it worth mentioning that the Lasry--Lions monotonicity condition on its own is in general not enough for the global well-posedness of the underlying master equation, unless a non-degenerate idiosyncratic noise (or stronger convexity assumption on the data) is also present and the corresponding Hamiltonians are separable in the momentum and measure variables, i.e. they possess a decomposition of the form
\begin{equation}\label{hyp:sep}
H(x,\mu,p):=H_0(x,p) - F(x,\mu),
\end{equation}
for some $H_0:\R^d\times\R^d\to\R$ and $F:\R^d\times\sP_2(\R^d)\to\R$ (where the state space of the agents is $\R^d$ and $\sP_2(\R^d)$ stands for the set of Borel probability measures with finite second moment, supported on $\R^d$, describing the agents' distribution). 

\indent Displacement monotonicity (which stems from the notion of displacement convexity arising in optimal transport, \cite{McC}) is an alternative condition which guarantees the existence and uniqueness of classical solutions to the master equation. Prior to using this condition in the context of master equations, under different names (as `weak monotonicity' or `$L$-monotonicity') this condition has already appeared in works on MFG (see \cite{Ahu,AhuRenYan} and \cite[Section 3.4.3]{CD1}) and on FBSDEs of McKean--Vlasov type (see \cite{CarDel15}). This condition turned out to be sufficient in the case of deterministic potential master equations in the lack of the regularizing effect of the idiosyncratic noise (\cite{TwoAuthor2022,BenGraYam}), or for a general class of non-separable Hamiltonians in the presence of non-degenerate idiosyncratic noise (\cite{FourAuthor}). In the presence of Lasry--Lions monotonicity and non-degenerate idiosyncratic noise, for separable Hamiltonians suitable notions of weak solutions have been proposed if the data are not regular enough (\cite{MouZha23,Ber}). In such cases, we still have uniqueness of the MFG Nash equilibria.

\indent  Recently, new monotonicity conditions were proposed beyond the Lasry--Lions and displacement monotonicity regimes. First, so-called {\it anti-monotonicity} conditions were introduced in \cite{MouZha22} (see also \cite{BanMes} in this context). This regime includes coupling functions which are not monotone in the Lasry--Lions or displacement sense. In this work it was shown that if the data are sufficiently anti-monotone in an appropriate sense, one has uniqueness of Nash equilibria and the global well-posedness of the master equation holds, from the point of view of classical solutions. Second, the recent works \cite{GraMes22, GraMes22b} proposed a new notion of monotonicity condition (in general in dichotomy with all the previously mentioned ones), which is also a sufficient one for the global well-posedness of the master equation.

\indent When the uniqueness of the MFG Nash equilibria does not hold, the classical well-posedness theory for the master equation breaks down in finite time and it is a great challenge to define suitable notions of weak solutions, which may help selecting specific equilibria of the game. In this direction it worth mentioning the recent breakthrough \cite{CecDel} which proposes a notion of weak solution (in the spirit of entropy solutions) for potential MFG master equations in the presence of non-degenerate idiosyncratic noise. It has been pointed out in \cite{GraMes22} that weak solutions in entropy sense (although different from the ones in \cite{CecDel}) might in general not select MFG Nash equilibria of the underlying game. In the very recent work \cite{MouZhang23} a partial order has been proposed on the set of Nash equilibria arising in a class of non-potential MFGs. The value functions corresponding to the minimal and maximal Nash equilibria are proved to satisfy the master equation in a certain weak sense. In particular, when these two value functions coincide, this becomes the unique solution to the master equation.

\indent As discussed above, the Lasry--Lions monotonicity condition on the data without the presence of a non-degenerate idiosyncratic noise in general cannot guarantee the uniqueness of solutions to the MFG system (see the discussion in \cite{GraMes22b}) and so, the existence of a classical solution to the corresponding master equation. 
In the lack of non-degenerate idiosyncratic noise, the literature discusses two important classes of examples: purely deterministic problems and problems driven only by a common noise. In the case of deterministic Lasry--Lions monotone MFG systems, \cite[Theorem 1.8]{CarPor20} presents a uniqueness result under the additional assumption that the measure component is essentially bounded. MFG systems and master equations driven by common noise only have been recently investigated in the series of interesting works. 
We refer the reader to \cite{CarSou22-sima,CarSou22-aap}. Further progress in this direction was achieved in \cite{CarSeeSou23}, where the authors consider common noise and degenerate idiosyncratic noise. In \cite{CarSou22-sima}, a notion of weak solution for the master equation is obtained in the spirit of monotone solutions proposed in \cite{Ber}. The study of MFG with common noise goes back to the works \cite{Ahu,CarDelLac}. Interestingly, already in these early works it has been discussed that additional convexity properties of the value function can render a stronger notion of solutions to MFG with common noise (than the ones in \cite{CarSou22-aap,CarSeeSou23}). It is well-known now that the convexity of the value function in the state variable is strongly linked to the displacement monotonicity of the data (see \cite{TwoAuthor2022,FourAuthor, MesMou, GraMes22b}).

\indent To the best of our knowledge, there are only very few works in the literature studying the global existence and uniqueness of classical solutions to the master equation in lack of non-degenerate idiosyncratic noise: \cite{TwoAuthor2022} considers potential deterministic master equation in the case of separable Hamiltonians and displacement convexity; \cite{GraMes22} studies a class of deterministic master equations in the presence of a different monotonicity condition; a particular dimension reduction technique and the associated monotonicity conditions allowed the authors of \cite{LasLioSee} to obtain global classical solutions to the deterministic master equation. Finally, in \cite{CarSou22-sima} the authors obtain weak monotone solutions to a class of time independent master equations both in the deterministic setting and driven by common noise  only (clearly, without any displacement monotonicity assumptions on the data, and in the case of separable Hamiltonians).

\medskip

\indent Our objective in this manuscript is to show the global existence and uniqueness of classical solutions to the master equation in the lack of idiosyncratic noise and the presence of displacement monotone data. Our result cover both the deterministic problem and the one driven purely by a common noise. The main actor of our study is the master equation
\begin{equation}\label{eq:master}
\left\{
\begin{array}{rcll}
-\partial_t V(t,x,\mu) + H(x,\mu,\pa_x V) - (\m NV)(t,x,\mu) - \frac{\beta^2}{2} \Delta_{{\rm{com}}} V(t,x,\mu) &=& 0,& {\rm{in}}\ (0,T)\times\R^d\times\sP_2(\R^d), \\ 
V(T, x, \mu) &=& G(x, \mu), & {\rm{in}}\  \R^d\times\sP_2(\R^d)
\end{array} 
\right.
\end{equation}
where
\[
(\m NV)(t,x,\mu) :=  -\int_{\R^d} \pa_\mu V(t,x,\mu,\ti x) \cdot \pa_pH(\ti x, \mu, \pa_x V(t, \ti x, \mu)) \dd\mu(\ti x)
\]
and 
\[
\Delta_{{\rm{com}}}V (t,x,\mu) &:= \tr(\pa_{x x} V(t,x,\mu)) + \int_{\R^d} \tr(\pa_{\ti x \mu} V(t,x,\mu,\ti x)) \dd\mu(\ti x) + 2 \int_{\R^d} \tr(\pa_{x \mu} V(t,x,\mu,\ti x)) \dd\mu(\ti x) \\
&+ \iint_{\R^{d}\times\R^d} \tr(\pa_{\mu \mu} V(t,x,\mu,\ti x, \bar x)) \dd\mu(\ti x) \dd\mu(\bar x).
\]
Here $T>0$ is the time horizon of the game, $\beta\in \R$ stands for the intensity of the common noise represented by a Brownian motion $(B^{0}_t)_{t\in[0,T]}$ on $\R^d$, $H:\R^d\times\sP_2(\R^d)\times\R^d\to\R$ and $G:\R^d\times\sP_2(\R^d)\to\R$ are the Hamiltonian and the final cost function, respectively.

\begin{defin}
A function $V: (0,T)\times\R^d\times\sP_2(\R^d) \to \R$ is said to be a classical solution to the master equation if all of the derivatives that appear in the equation exist and are continuous (with respect to Euclidean distance and $W_1$) and $V$ satisfies the master equation pointwise. 
\end{defin}

The master equation \eqref{eq:master} is strongly linked to the following mean field games system: for any $t_0\in (0,T)$,

\begin{equation}\label{eq:SPDE}
\left\{
\begin{array}{lll}
\displaystyle \dd u(t, x)&\displaystyle =  - \Big[\tr\big(\frac{\beta^2}{2} \pa_{xx} u(t,x) + \beta\partial_x v^\top(t,x)\big) -H(x,\rho_t,\pa_x u(t,x))\Big]\dd t +  v(t,x)\cdot \dd B_t^0, & {\rm{in}}\ (t_0,T)\times\R^d,\\
\displaystyle \dd \rho(t,x)&\displaystyle =  \Big[\frac{\beta^2}{2}\tr\big( \pa_{xx} \rho(t,x)\big) + {\rm{div}}(\rho(t,x) \pa_p H(x, \rho_t, \pa_x u(t,x)))\Big]\dd t-\beta\partial_x\rho(t,x)\cdot \dd B_t^0, & {\rm{in}}\ (t_0,T)\times\R^d,\\
\displaystyle&\displaystyle \rho(t_0,\cdot) = \mu,\  u(T,x) = G(x, \rho(T,\cdot)), & {\rm{in}}\ \R^d.
\end{array}
\right.
\end{equation}
The solution to \eqref{eq:SPDE} is a triple $(\rho, u, v)$, progressively measurable with respect to the filtration generated by the common noise $B^0$, which serves formally as the system of generalized characteristics for \eqref{eq:master}. We note that if $\beta\neq0$, then $\rho(t,\cdot,\omega)$ is a random probability measure. Conversely, the solution $V$ to the master equation \eqref{eq:master} also serves as the decoupling field for this forward-backward SPDE system, i.e.
\begin{equation}\label{eqn: MFSrepform}
u(t,x,\omega) = V(t,x, \rho(t,\cdot,\omega)).
\end{equation}

\medskip

{\bf The description of our results.} As our main result (Theorem \ref{thm:main}) we show the global in time existence and uniqueness of a classical solution to \eqref{eq:master}, by assuming that $H$ and $G$ satisfy the displacement monotonicity and suitable regularity conditions. The roadmap to the proof of our main result is similar in spirit to the one used in \cite{FourAuthor,MouZha22}, but several new ideas were necessary to fulfill this because of the lack of the idiosyncratic noise. 

\indent Let us discuss the main similarities and differences in the two approaches. First, the heart of our analysis is the a priori propagation of the displacement monotonicity: if $V$ is a classical solution to \eqref{eq:master} and $H$ and $G$ are displacement monotone, so is $V(t,\cdot,\cdot)$. Displacement monotonicity will readily imply that $\pa_x V(t,x,\cdot)$ is Lipschitz continuous with respect to the metric $W_2$ (with a Lipschitz constant depending on the data and on $\|\pa_{xx}V\|_{L^\infty}$). 
It is well-known (see \cite{Carmona2018,CarCirPor20,BerLasLio:24}) that the master equation is well-posed for short time if the data is regular enough (without any monotonicity assumptions). The short time horizon depends on the Lipschitz constant of $\partial_x V(t,\cdot,\cdot)$ (in the metric $W_2$ for the measure variable). However, for the global well-posedness of classical solutions to \eqref{eq:master}, the uniform a priori Lipschitz continuity of $\pa_xV(t,x,\cdot)$ in the metric $W_1$ is essential (see the discussion in \cite{FourAuthor}). To show that this $W_1$-Lipschitz constant is a priori bounded, we use two arguments. First, the uniform a priori estimates on $\|\pa_{xx}V\|_{L^\infty}$ are a consequence of the semi-concavity bounds (a result of classical optimal control arguments) and convexity (a consequence of the displacement monotonicity) of $V(t,\cdot,\mu)$. To obtain the necessary a priori bounds on $\|\pa_{\mu x}V\|_{L^\infty}$ we rely on several representation formulas via suitable FBSDE systems. Although these representation formulas are similar in spirit to the ones used in \cite{MouZha23,FourAuthor}, we need to work with different systems of FBSDEs. During this process, we show also that -- similarly as in \cite{MouZha23,FourAuthor} -- the small time horizon depends on the $W_2$-Lipschitz constant of $\partial_x V$. 

\indent This approach represents one of the major differences with the work \cite{FourAuthor}. Let us elaborate more on this. Indeed, we can observe that the FBSDE systems in \cite{FourAuthor} (see for instance the \cite[System (2.24)]{FourAuthor}) are not natural if the intensity of the idiosyncratic noise is taken to be zero. Therefore, instead we will be working with FBSDE systems of Pontryagin type, where the natural variables are the state and the momentum (instead of the state and optimal value, as in \cite{FourAuthor}). Such FBSDE systems have already been introduced and studied in \cite{ChaCriDel22,Carmona2018,MouZha22} for the study of master equations. It will become the classical forward-backward Hamiltonian system in case of deterministic problems. This subtlety has been emphasized in \cite[Section 5.2]{ChaCriDel22}: in the case when non-degenerate idiosyncratic noise is present optimal paths may be characterized by solutions of FBSDEs, where the value function is represented as the decoupling field of the forward-backward system (similarly to the approach used also in \cite{FourAuthor}); on the contrary, when the idiosyncratic noise is degenerate but additional convexity is present on the data (which is provided in our case by the displacement monotonicity), the natural characterization of the optimal paths may be obtained via the stochastic Pontryagin principle, where the decoupling field of the FBSDE system is understood as the gradient of the value function of the optimization problem.

\indent The FBSDE systems used in \cite{FourAuthor} required stricter assumptions on the data. For instance, as we can see in \cite[Assumpotions 3.1 and 3.2]{FourAuthor}, $G$ was assumed to be globally Lipschitz continuous (with respect to the metric $W_1$ in the measure variable) and $H$ was assumed to be Lipschitz continuous in all three variables (locally in the momentum variable, but globally in the state and measure variables). These actually imposed that $\pa_x G,\pa_\mu G$ are uniformly bounded (in $\R^d\times\sP_2(\R^d)$) and in case of $H$, $\pa_x H,\pa_pH$ are uniformly bounded in $\R^d\times\sP_2(\R^d)\times B_R(0)$ and $\pa_\mu H$ uniformly bounded in $\R^d\times\sP_2(\R^d)\times\R^d\times B_R(0)$ (with constants possibly depending on $R>0$). In contrast to these, we improve these assumptions in the way that we require only $\pa_x G, \pa_xH,\pa_p H$ to be uniformly Lipschitz continuous (see Assumptions \ref{assum: G} and \ref{assum: H} below). As a result of the assumption in \cite{FourAuthor} the value function was such that both $\pa_x V$ and $\pa_\mu V$ are uniformly bounded, while this will not be the case in our work (allowing for instance both $\pa_x V$ and $\pa_\mu V$ to have linear growth in $x$ at infinity).

\indent Another significant difference is that in contrast to \cite{FourAuthor}, in the current manuscript the displacement monotonicity conditions imposed both on $G$ and $H$ are first order conditions (in \cite{FourAuthor} the displacement monotonicity condition imposed on $H$ was a delicate second order condition). As a result of this, we have sharpened the assumptions on the data significantly. This condition on $H$ was proposed first in \cite{MesMou}, and we refer to a discussion therein, which details that this in general does not imply the second order monotonicity condition imposed on $H$ in \cite{FourAuthor}, even in the case of smooth Hamiltonians. Because of this, the analysis that we use here is significantly different from the one in \cite{FourAuthor}, as for the propagation of the displacement monotonicity we avoid differentiating the master equation itself. We work directly at the level of the FBSDE and MFG systems. Our inspiration for these techniques comes from \cite{MesMou}.

\indent It is not hard to see that adding an additional idiosyncratic Brownian noise (with a constant intensity) would result in essentially the same analysis as the one present in this manuscript. The Pontryagin principle used here is very similar in spirit to the analysis on the deterministic Hamiltonian system in \cite{TwoAuthor2022}, however, the results there (as they rely both on the separable Hamiltonian and potential game structure) cannot imply our results if $\beta=0$. Therefore, the results of our manuscript unify and generalize the results of \cite{TwoAuthor2022} and \cite{FourAuthor} significantly.

\medskip

\indent The structure of the rest of the paper is simple. In Section \ref{sec:2} we present some notations and the necessary assumptions on the data. Section \ref{sec:3} contains the classical semi-concavity estimates and convexity results for the master field (i.e. the candidate solution to the master equation). Here we discuss the propagation of the displacement monotonicity and its consequences as well. Section \ref{sec:4} contains the technical results on the representation formula for $\pa_{\mu x}V$ which yields the crucial a priori $W_1$-Lipschitz estimate for $\pa_x V$. Section \ref{sec:5} contains a by now standard argument describing how to extend the local in time well-posedness theory for the master equation, in case of sufficient a priori estimates.

\section{Notations, setup and assumptions}\label{sec:2}

Let us fix an arbitrary finite time horizon $T>0$. Throughout the paper we will use the following product filtered probability space on $[0, T]$: 
\begin{eqnarray*}
\Omega := \Omega_0 \times \Omega_1,\quad \mathbb F := \{\mathcal{F}_t\}_{0\leq t\leq T} := \{\mathcal{F}^0_t \otimes \mathcal{F}_0^1\}_{0\leq t\leq T},\quad \mathbb P := \mathbb P_0\otimes \mathbb P_1,\quad \mathbb E:= \mathbb E^\mathbb P.
\end{eqnarray*}
Here, for $\omega = (\omega^0, \omega^1)\in \Omega$, $B^0(\omega) = B^0(\omega^0)$ is a $d$-dimensional Brownian motions; $\mathbb F^0=\{\mathcal{F}^0_t\}_{0\leq t\leq T}$ is generated by $B^0$; we assume $\mathcal{F}^1_0$ has no atom.
Let $(\tilde \Omega_1, \tilde {\mathcal F}_0^1, \tilde {\mathbb P}_1)$ be a copy of the probability space $(\Omega_1,   \mathcal F_0^1, \mathbb P_1)$ and define the larger filtered probability space by
\begin{eqnarray*}
\tilde \Omega := \Omega\times \tilde\Omega_1 ,\quad \tilde{\mathbb F} = \{\tilde {\mathcal{F}}_t\}_{0\leq t\leq T} := \{\mathcal {F}_t \otimes \tilde {\mathcal {F}}^1_0\}_{0\leq t\leq T},\quad \tilde {\mathbb P} := \mathbb P\otimes \tilde{\mathbb P}_1,\quad \tilde {\mathbb E}:= \mathbb E^{\tilde {\mathbb P}}.
\end{eqnarray*}
Given an $\mathcal{F}_t$-measurable random variable $\xi(\tilde \omega) = \varphi(\omega^0, \omega^1)$, $\tilde\omega=(\omega^0, \omega^1, \tilde \omega^1)\in \tilde \Omega$, we see that $\tilde \xi(\tilde \omega) := \varphi(\omega^0, \tilde \omega^1)$ is a conditionally independent copy of $\xi$, conditional on $\mathcal{F}^0_t$ under $\tilde {\mathbb P}$.

Throughout the paper we will use the filtered probability space $(\Omega, \mathbb F, \mathbb P)$. However, when conditionally independent copies of random variables or processes are needed, we will tacitly use their extensions to the larger space $(\tilde \Omega, \tilde {\mathbb F}, \tilde {\mathbb P})$ 
without mentioning. For any $t_0\in [0,T]$ we define $B_t^{0,t_0}:=B_t^0-B_{t_0}^0$. For any $\mathcal{G}\subset \mathcal{F}$, we use $L^2(\mathcal{G})$ to denote the set of $\mathcal{G}$-measurable and 2-integrable random variables.

\subsection{Elements of analysis and calculus on the Wasserstein space}\label{sect:Wasserstein}

Let $\sP(\R^d)$ be the set of Borel probability measures supported in $\mathbb R^d$. For any $ q\geq 1$ and any measure $\mu\in \sP(\R^d)$, we set $M_q(\mu):=\left( \int_{\mathbb R^d} |x|^q \dd\mu(x)\right)^{1\over q}$. Furthermore, let $\sP_q(\R^d):=\{\mu\in\sP(\R^d):~ M_q(\mu) < \infty\}$. For any $\mu,\nu\in \sP_q(\R^d)$, the $W_q$--Wasserstein distance  is defined as
$$
W_q(\mu, \nu) := \inf\left\{\iint_{\R^d\times\R^d}|x-y|^q\dd\gamma(x,y):\ \gamma\in\Pi(\mu,\nu)\right\}^{\frac1q},
$$
where $\Pi(\mu,\nu):=\left\{\gamma\in\sP_2(\R^d\times\R^d):\ (\pi_1)_\sharp\gamma=\mu,\ (\pi_2)_\sharp\gamma=\nu\right\},$ and $\pi_1,\pi_2:\R^d\times\R^d\to\R^d$ stand for the canonical projections, i.e. $\pi_1(x,y)=x$ and $\pi_2(x,y)=y$.

According to the terminology in \cite{AGS}, the Wasserstein gradient of a function $U: \sP_2(\R^d) \to \R$ at $\mu$, is an element $\pa_\mu U(\mu,\cdot)\in \overline{\nabla C_c^\infty(\mathbb R^d)}^{L^2_\mu}$ (the closure of gradients of $C_c^\infty$ functions in $L^2_\mu(\R^d;\R^d)$) and so, it is a priori defined $\mu$--almost everywhere. The theory developed in \cite{Lions, WuZhang17, GanTud19, CarPor20}, 
shows that $\pa_\mu U(\mu,\cdot)$ can be characterized by the property 
\begin{align}
\label{WassersteinDerivativeDef}
U(\cL_{\xi +  \eta}) - U(\mu) = \E\big[\langle \pa_\mu U(\mu, \xi), \eta \rangle \big] + o(\|\eta\|_2), \ \forall\ \xi,\eta,\ {\rm{with}}\ \cL_\xi=\mu.
\end{align}
Let $\cC^0(\sP_2(\R^d))$ denote the space of $W_2$--continuous functions $U: \sP_2(\R^d) \to \R$. For $k \in \{1,2\}$ we next define a subset of $\cC^k(\sP_2(\R^d))$, referred to as functions of {\it full $\cC^k$} regularity in \cite[Chapter 5]{CD1}),  as follows. By $\cC^1(\sP_2(\R^d)),$ we mean the space of functions $U\in \cC^0(\sP_2(\R^d))$ such that $\pa_\mu U$ exists for all $\mu\in\sP_2(\R^d)$ and it has a unique jointly continuous extension to $\sP_2(\R^d)\times \R^d$, which we continue to denote by 
$$ \mathbb R^d\times\sP_2(\R^d)\ni (\tilde x, \mu) \mapsto \pa_\mu U (\mu, \tilde x)\in \mathbb R^d.$$ 
Similarly, $\cC^2(\sP_2(\R^d))$ stands for the space of functions $U\in \cC^1(\sP_2(\R^d))$ such that the global version of $\partial_\mu U$ is differentiable in the sense that all the following maps exist and have unique jointly continuous extensions
\begin{align*}
&\mathbb R^d\times\sP_2(\R^d)\ni (\tilde x, \mu) \mapsto  \pa_{\tilde x\mu}U(\mu, \tilde x) \in \mathbb R^{d\times d}\ \ {\rm{and}}\\ 
&\mathbb R^{2d}\times\sP_2(\R^d)\ni (\tilde x, \bar x,  \mu) \mapsto  \pa_{\mu\mu}U(\mu, \tilde x, \bar x) \in \mathbb R^{d \times d}.
\end{align*}
We define similarly the spaces $\cC^1(\R^d\times\sP_2(\R^d))$ and $\cC^{2}(\mathbb R^d\times\sP_2(\R^d))$. In particular $\cC^{2}(\mathbb R^d\times\sP_2(\R^d))$ is the space of continuous functions $U:\mathbb R^d\times\sP_2(\R^d)\to\R$ satisfying the following

(i) $\pa_xU, \pa_{xx}U$ exist and are jointly continuous on $\R^d\times\sP_2(\R^d)$;

(ii) The following maps exist and have unique jointly continuous extensions
\begin{align*}
&\mathbb R^{2d}\times\sP_2(\R^d)\ni (x, \tilde x, \mu) \mapsto \pa_\mu U(x, \mu, \tilde x)\in \mathbb R^d\ \ {\rm{and}}\\
&\mathbb R^{2d}\times\sP_2(\R^d)\ni (x, \tilde x, \mu) \mapsto\pa_{x\mu}U(x, \mu, \tilde x) \in \mathbb R^{d \times d};
\end{align*}

(iii) Finally, the following maps exist and have unique jointly continuous extensions 
\begin{align*}
&\mathbb R^{2d}\times\sP_2(\R^d)\ni (x, \tilde x, \mu) \mapsto \pa_{\tilde x\mu}U(x, \mu, \tilde x) \in \mathbb R^{d \times d} \ \ \text{and}\\
&\mathbb R^{3d}\times\sP_2(\R^d)\ni (x, \tilde x, \bar x, \mu) \mapsto \pa_{\mu\mu}U(x, \mu, \tilde x, \bar x) \in \mathbb R^{d \times d}.
\end{align*}

We underline that for notational conventions, we always denote the `new spacial variables' appearing in Wasserstein derivatives with `tilde' symbols (for first order Wasserstein derivatives), with `bar' symbols (for second order Wasserstein derivatives) and so on, and we place them right after the corresponding measures variables. For example, when $U:\R^d\times\sP_2(\R^d)\times\R^d\to\R$ is typically evaluated as $U(x,\mu,p)$, we use the notations $\partial_\mu U(x,\mu,\tilde x,p)$, $\partial_{\tilde x\mu} U(x,\mu,\tilde x, p)$, $\partial_{\mu\mu} U(x,\mu,\tilde x,\bar x, p)$, and so on. This convention will be carried through to compositions with random variables too, for example $\partial_\mu U(x,\mu,\tilde\xi,p)$, when $\tilde\xi$ is an $\R^d$-valued random variable.

\subsection{Displacement Monotonicity} Following \cite{FourAuthor}, we can recall the displacement monotonicity condition for the final cost function $G$.

\begin{defin}\label{def:DM_G}
Let $G: \R^d \times \sP_2(\R^d) \to \R$. Assume $G$ is differentiable in the first variable and $\pa_xG$ is continuous on $\R^d\times\sP_2(\R^d)$. We say that $G$ is \emph{displacement monotone} if for all $\xi_1,\xi_2 \in L^2(\mathcal{F}_0)$
\begin{equation}\label{eq:displacementG}
\int_{\Omega}\left[\pa_xG(\xi_1(\omega),\mathcal{L}_{\xi_1})-\pa_xG(\xi_2(\omega),\mathcal{L}_{\xi_2})\right]\cdot (\xi_1(\omega)-\xi_2(\omega))\dd\mathbb{P}(\omega)\geq 0.
\end{equation} 
\end{defin}

\begin{rmk}
Assume further that $\partial_xG\in \cC^1(\R^d\times\sP_2(\R^d))$ and $\pa_{x\mu}G,\pa_{xx}G$ are bounded. Then $G$ is displacement monotone if and only if for all $\xi,\eta\in L^2(\mathcal{F}_0)$
\[
\iint_{\Omega\times\Omega}\left[ \langle \pa_{x\mu}G\(\xi(\omega), \L_\xi , \tilde \xi(\ti \omega)\) \tilde\eta(\ti \omega) , \eta(\omega) \rangle +
\langle \pa_{xx}G\(\xi(\omega), \L_\xi\) \eta(\omega), \eta(\omega) \rangle\right] \dd\P(\omega)\dd\P(\ti \omega)\geq 0,
\]
where $\tilde\xi, \tilde\eta$ are independent copies of $\xi$ and $\eta$, respectively.
\end{rmk}

We recall the following lemma from \cite[Lemma 2.3]{MesMou}.

\begin{lem}\label{lem:dismon convex}
Suppose that $G$ is displacement monotone and $\partial_xG\in \cC^1(\R^d\times\sP_2(\R^d))$. Then $\pa_{xx}G(x, \mu) \geq 0$ for all $(x,\mu)\in\R^d\times\sP_2(\R^d)$. 
\end{lem}

Following \cite{MesMou}, we can recall the displacement monotonicity condition for the Hamiltonian $H$.
\begin{defin}\label{def:DM_H}
Let $H: \R^d \times \sP_2(\R^d) \times \R^d \to \R$. Assume that $H$ is differentiable in the $x$ and $p$ variables and $\pa_xH,\pa_pH$ are continuous on $\R^d\times\sP_2(\R^d)\times\R^d$. We say that $H$ is \emph{displacement monotone} if for all $\xi_1,\xi_2,p_1,p_2 \in L^2(\mathcal{F}_0)$
\begin{align*}
&  \iint_{\Omega \times \Omega}\bigg[\Big\langle -\pa_{x} H(\xi_1(\omega), \mathcal{L}_{\xi_1}, p_1(\omega)) + \pa_{x} H(\xi_2(\omega), \mathcal{L}_{\xi_2},p_2(\omega)),\; \xi_1(\omega)-\xi_2(\omega)\Big\rangle  \bigg]\\
&+ \bigg[\Big\langle \pa_{p} H(\xi_1(\omega), \mathcal{L}_{\xi_1}, p_1(\omega)) - \pa_{p} H(\xi_2(\omega), \mathcal{L}_{\xi_2},p_2(\omega)),\; p_1(\omega)-p_2(\omega)\Big\rangle  \bigg]\dd\P(\omega)\geq 0.
\end{align*} 
\end{defin}

\begin{rmk}
Assume further that $H\in \cC^2(\R^d\times\sP_2(\R^d)\times\R^d)$ and $\pa_{x\mu}H,\pa_{xx}H,\pa_{pp}H,\pa_{p\mu}H$ are bounded. Suppose that
\begin{align}\label{ineq:disp2}
&  \iint_{\Omega\times\Omega}\left[ \langle \pa_{x\mu}H\(\xi(\omega), \L_\xi , \tilde \xi(\ti \omega),p(\omega)\) \tilde \eta(\ti \omega) , \eta(\omega) \rangle +
\langle \pa_{xx}H\(\xi(\omega), \L_\xi,p(\omega)\) \eta(\omega), \eta(\omega) \rangle\right] \dd\P(\omega)\dd\P(\ti \omega)\\
&\nonumber\leq -\frac{1}{4} \iint_{\Omega\times\Omega}\bigg| [\pa_{pp}H(\xi(\omega),\mathcal{L}_{\xi},p(\omega))]^{-\frac{1}{2}}\pa_{p\mu}H(\xi(\omega),\mathcal{L}_{\xi},\tilde \xi(\tilde \omega),p(\omega))\tilde \eta(\tilde\omega)\bigg|^2\dd\P(\omega)\dd\P(\tilde\omega),
\end{align} 
for all $\xi,\eta,p\in L^2(\mathcal{F}_0)$, where $\tilde\xi,\tilde\eta$ are independent copies of $\xi$ and $\eta$, respectively. Then $H$ is displacement monotone.

We would like to underline that the main standing assumption on $H$ in \cite{FourAuthor} was the inequality \eqref{ineq:disp2}. In contrast to this, in the current manuscript we impose the inequality from Definition \ref{def:DM_H}, which is considerably weaker, and in general does not imply \eqref{ineq:disp2} (cf. \cite{MesMou}).

\end{rmk}
\subsection{Standing assumptions}

We will always assume that $G, H$ are displacement monotone in the sense of Definition \ref{def:DM_G} and Definition \ref{def:DM_H}, respectively. We also make the following regularity assumptions. 

\begin{assum}\label{assum: G} We assume that 

\begin{enumerate}
\item $G\in \cC^2(\R^d\times\sP_2(\R^d))$ and that there exists $L^G$ such that
 $\abs{\pa_{xx} G},\abs{\pa_{\mu x} G}\leq L^G$ and $G\geq -L^G$.
 
 \item All second order derivatives of $G$ are uniformly continuous with respect to the space and measure variables (the latter with respect to $W_2$).
 \end{enumerate}
\end{assum}

\begin{assum}\label{assum: H} We assume that

\begin{enumerate}
\item $H\in \cC^2(\R^d\times\sP_2(\R^d)\times\R^d)$ and there exists $L^H>0$ such that
 $\abs{\pa_{px} H},\abs{\pa_{xx} H}, \abs{\pa_{pp} H}, \abs{\pa_{x\mu} H}, \abs{\pa_{p\mu} H}\leq L^H$ and 
 \begin{equation}\label{eq:Lbdd}
 \langle\pa_pH(x,\mu,p),p\rangle-H(x,\mu,p)\geq -L^H,\ \forall\ (x,\mu,p)\in\R^d\times\sP_2(\R^d)\times\R^d.
 \end{equation}

\item $\pa_{pp} H \geq c_0 I$ for some $c_0 > 0$. 

 \item All second order derivatives of $H$ are uniformly continuous with respect to the space and measure variables (the latter with respect to $W_2$).
\end{enumerate}
\end{assum}

We underline furthermore that these assumptions are weaker than those in \cite[Assumption 3.1]{FourAuthor} in the sense that we do not require the uniform boundedness of $\pa_x G$ or $\pa_\mu G$ and moreover we need much less regularity assumptions on the data. However, we impose strong convexity condition for the Hamiltonian in the $p$ variable because of the absence of the idiosyncratic noise. Also, as mentioned above, instead of the second order displacement monotonicity condition, as imposed on $H$ in \cite{FourAuthor}, here we impose the first order condition. We denote by $L^G_2$ the Lipschitz constant of $\pa_x G$ with respect to space and $W_2$ norm in measure. Note that $|{\pa_{\mu x} G}|\leq L^G_1$ implies that $\pa_{x} G$ is Lipschitz with respect to $W_1$ norm in measure with the Lipschitz constant $L^G_1$, and so $L^G_2 \leq L^G_1$.

\begin{defin}
A constant $C$ is said to be universal if it depends only on the above quantities ($L^G, L^H,$ and $c_0$) and $T$. 
\end{defin}

\subsection{The roadmap of the well-posedness theory}

The proof of our main theorem (Theorem \ref{thm:main}) will go as follows. First we assume that we have a smooth solution, $V$, to the master equation \eqref{eq:master}. We show that $V$ is displacement monotone. This is a consequence of the displacement monotonicity assumption on $H$ and $G$ and the proof of this is inspired by the results in \cite{MesMou}. 

\smallskip

The heart of our analysis is to show that $\pa_{x} V$ is Lipschitz continuous (with respect to the Euclidean norm in $x$ and the $W_1$ distance in $\mu$) where the Lipschitz constant is universal. To achieve this, we will proceed as follows. First, the uniform boundedness of $\partial_{xx}V$ will be implied by semi-concavity estimates on $V$ in the space variable (which comes from classical optimal control arguments) and its convexity in the space variable (which is implied immediately by displacement monotonicity, c.f. Lemma \ref{lem:dismon convex}). Second, for the Lipschitz continuity of $\partial_x V$ in the measure variable we will first show that this is Lipschitz continuous with respect to the $W_2$ distance. Instead of the approach used in \cite{FourAuthor}, for this result we follow the ideas from \cite{MesMou}. From this, together with the uniform bounds on $\partial_{xx}V$, we will be able to see that a certain system of FBSDEs is well-posed and that this system gives us a representation formula for $\pa_{x\mu}V$. From here we deduce that $\pa_{x\mu}V$ is bounded by a universal constant which is equivalent to $\pa_x V$ being $W_1$-Lipschitz continuous (with the same universal constant).

\smallskip

Once these a priori bounds are proven, the well-posedness of the master equation will follow easily.

\section{Semi-concavity and displacement monotonicity of the master field}\label{sec:3}

\subsection{Semi-concavity of value functions}
\begin{lem}\label{lem: representationformula}
Assume $G,H$ satisfy Assumptions \ref{assum: G} and \ref{assum: H}. Let $V$ be a classical solution to the master equation with bounded $\pa_{xx}V$ and $\pa_{x\mu}V$ on $[0,T]$. Fix some $t_0 \in [0,T]$ and $\mu \in {\sP}_2(\R^d)$. Then there exists a $\mathbb F^{B^{0,t_0}}$-progressively measurable stochastic probability measure flow $(\rho_t)_{t\in[t_0,T]}$ solving 
\begin{equation}\label{eq:rhoV}
 d\rho(t,x) =  \Big[\frac{\beta^2}{2}\tr\big( \pa_{xx} \rho(t,x)\big) + {\rm{div}}(\rho(t,x) \pa_p H(x, \rho_t, \pa_x V(t,x,\rho_t)))\Big]\dd t-\beta\partial_x\rho(t,x)\cdot \dd B_t^0, \quad {\rm{in}}\ (t_0,T)\times\R^d
\end{equation}
with $\rho_{t_0} = \mu$ such that 
\begin{equation}\label{eq:VV}
V(t_0,x,\mu) = \inf_{\alpha} \E\left\{G(X_T^\alpha, \rho_T) + \int_{t_0}^T L(X_t^\alpha, \rho_t, \alpha_t) \dd t \right\},
\end{equation}
where $L$ is the Legendre transformation of $H$, the infimum is taken over the set of admissible controls which are progressively measurable and square integrable, and $X_t^\alpha$ satisfies the SDE 
$$\dd X_t^\alpha = \alpha_t \dd t + \beta \dd B_t^{0},\ {\rm{with}}\ X_{t_0}^\alpha = x.$$

\end{lem}

\begin{proof}
This is essentially a folklore result discussed in \cite[Remark 2.10 part (ii)]{FourAuthor}.
\end{proof}

\begin{prop}\label{cor:Vsemi-concave}

Assume that the assumptions in Lemma \ref{lem: representationformula} hold. 
Then $V$  is semi-concave in the $x$-variable with a semi-concavity constant $(1+T)C$. 
\end{prop}

The results of this proposition are certainly well-known for experts (see for instance \cite{Cannarsa04} for the deterministic setting, i.e. when $\beta=0$, and \cite{BuckCanQui} for a similar stochastic control problem). However we were unable to find a reference that matches our exact assumptions. Hence for completeness we reprove it. 

\begin{proof}[Proof of Proposition \ref{cor:Vsemi-concave}]
By (1) in Assumption \ref{assum: H} (specifically the bound on $\pa_{xx} H$) and (1) in Assumption \ref{assum: G}  we have that $L$ and $G$ are semi-concave (in $x$) with a universal constant. Note that \eqref{eq:Lbdd} implies that $L$ is bounded from below. Fix some $(t_0,x,\mu)\in [0,T]\times \R^d\times \sP_2(\R^d)$. Let $(\rho_t)_{t\in [t_0,T]}$ be given in \eqref{eq:rhoV} and let $(\alpha_t)_{t\in[t_0,T]}$ be an associated $\varepsilon$-optimal control for \eqref{eq:VV}. Fix some $\lambda \in \R^d$. Consider the exact same control $(\alpha_t)_{t\in[t_0,T]}$ as a proposed control for the problem initiated at $(t_0,x\pm\lambda)$. Note that the solution to
\[
\left\{
\begin{array}{ll}
\dd X_t^{\pm} = \alpha_t\dd t + \beta \dd B_t^{0}, & t\in(t_0,T), \\
X_{t_0}^{\pm} = x \pm \lambda
\end{array}
\right.
\] 
is simply $X_t^\alpha\pm \lambda$.
We get
\[
V(t_0, x + \lambda,\mu) \leq \E\left\{G(X_T^+,\rho_T) + \int_{t_0}^T L(X_t^+, \rho_t, \alpha_t) \dd t \right\}.
\]
By a symmetric argument we get
\[
V(t_0, x - \lambda,\mu) \leq \E\left\{G(X_T^-,\rho_T) + \int_{t_0}^T L(X_t^-,\rho_t, \alpha_t) \dd t \right\}
\]
and so
\begin{align*}
\frac{V(t, x + \lambda,\mu) + V(t,x - \lambda,\mu)}{2} - V(t,x,\mu) -\varepsilon &\leq \E\left\{ \frac{G(X_T^+,\rho_T) + G(X_T^-,\rho_T)}{2} - G(X_T,\rho_T) \right\}\\
&+ \E\left\{\int_{t_0}^T \frac{L(X_t^+, \rho_t, \alpha_t)+ L(X_t^-, \rho_t, \alpha_t)}{2} - L(X_t,\rho_t, \alpha_t) \dd t \right\} \\
&\leq C \abs{\lambda}^2 + TC\abs{\lambda}^2.
\end{align*}
By the arbitrariness of $\varepsilon>0$ we conclude as desired. 
\end{proof}

{\subsection{Propagation of displacement monotonicity and a priori $W_2$-Lipschitz continuity}

Suppose that $V$ is a classical solution to the master equation \eqref{eq:master} with bounded $\pa_{xx}V$ and $\pa_{x\mu}V$ on $[0,T]$. Under Assumption \ref{assum: H}, we have that $(t,x,\mu)\mapsto \pa_pH(x,\pa_xV(t,x,\mu),\mu)$ is continuous with respect to $t$ and globally Lipschitz continuous with respect to $x$ and the $W_1$ distance in $\mu$. Then the following McKean--Vlasov SDE admits a strong solution $X^\xi$
\begin{equation}\label{eq:defX}
X_t^{\xi}=\xi-\int_{t_0}^tD_pH(X_s^{\xi},\pa_xV(s,X_s^{\xi},\rho_s),\rho_s)\dd s+\beta B_t^{0,t_0},\quad t\in [t_0,T]
\end{equation}
where $\xi\in L^2(\mathcal{F}_{t_0})$ and $\rho_s:=\cL_{X_s^\xi|\mathcal{F}_s^0}$. We define 
\begin{equation}\label{eq:defYZ}
Y_t^\xi:=\pa_xV(t,X_t^\xi,\rho_t)\quad\text{and}\quad Z_t^{0,\xi}=\pa_{xx}V(t,X_t^\xi,\rho_t)+\tilde\E_{\mathcal{F}_t}[\pa_{x\mu}V(t,X_t^\xi,\rho_t,\tilde X_t^\xi)],\quad t\in [t_0,T].
\end{equation}
Then, we have that $(X^\xi,Y^\xi,Z^{0,\xi})$ is a strong solution to the following McKean--Vlasov FBSDE on $[t_0,T]$ associated to the master equation \eqref{eq:master}
\begin{equation}\label{eq:FBSDEnew}
\left\{
\begin{array}{l}
\displaystyle X_t^{\xi}=\xi-\int_{t_0}^t\pa_pH(X_s^{\xi},\rho_s,Y_s^{\xi})\dd s+\beta B_t^{0,t_0}, \\
\displaystyle Y_t^{\xi}=\partial_xG(X_T^{\xi},\rho_T)-\int_t^T \partial_xH(X_s^\xi,\rho_s,Y_s^\xi)\dd s-\beta\int_t^TZ_s^{0,\xi}\cdot \dd B_s^0.
\end{array}
\right.
\end{equation}

A crucial step in our analysis is the a priori propagation of the displacement monotonicity for the solution of the master equation. In contrast to the approach present in \cite{FourAuthor}, where the a priori propagation of the second order condition has been obtained via differentiating the master equation itself, here we follow a different route. We show the propagation of the first order condition, using the first order assumption (from Definition \ref{def:DM_H}) on $H$. This will be done by relying on the solution to the MFG system \eqref{eq:SPDE}. This idea is inspired from \cite{MesMou} to show uniqueness of solutions to the MFG system, where a similar analysis has been carried out in the absence of common noise.

\begin{prop}\label{prop:monotone}
Suppose that $G,H$ are displacement monotone and satisfy Assumptions \ref{assum: G} and \ref{assum: H} and that $V$ is a classical solution to the master equation with bounded $\pa_{xx}V$ and $\pa_{x\mu}V$ on $[0,T]$. Then for each fixed $t$, we have that $V(t, \cdot,\cdot)$ is displacement monotone. 
\end{prop}
\begin{proof} For any $t_0\in [0,T]$ and $\xi_i\in \mathbb L^2(\mathcal{F}_{t_0})$, $i=1,2$, let $X_t^i$ be the strong solution of the stochastic differential equation
\begin{equation}\label{eq:Xi}
X_t^i=\xi_i-\int_{t_0}^t\pa_pH(X_s^i,\rho_s^i,\pa_{x}V(s,X_s^i,\rho_s^i))\dd s+\beta B_t^{0,t_0},\quad i=1,2
\end{equation}
where $\rho_t^i:=\mathcal{L}_{X_t^i|\mathcal{F}_t^0}$. Define $u_i(t,x):=V(t,x,\rho_t^i)$ and $v_i(t,x):=\beta\mathbb E_{\mathcal{F}_t}[\pa_\mu V_i(t,x,\rho_t^i,\tilde X_t^{i})]$. Then $(\rho^i,u_i,v_i)$ is a classical solution of the mean field game system
\begin{equation}\label{eq:SPDEi}
\left\{
\begin{array}{lll}
\displaystyle \dd u_i(t, x)&\displaystyle =  - \Big[\tr\big(\frac{\beta^2}{2} \pa_{xx} u_i(t,x) + \beta\partial_x v_i^\top(t,x)\big) -H(x,\rho^i(t,\cdot),\pa_x u_i(t,x))\Big]\dd t +  v_i(t,x)\cdot \dd B_t^0, & (t_0,T)\times\R^d,\\
\displaystyle \dd \rho^i(t,x)&\displaystyle =  \Big[\frac{\beta^2}{2}\tr\big( \pa_{xx} \rho^i(t,x)\big) + {\rm{div}}(\rho^i(t,x) \pa_p H(x, \rho^i(t,\cdot), \pa_x u_i(t,x)))\Big]\dd t-\beta\partial_x\rho^i(t,x)\cdot \dd B_t^0, & (t_0,T)\times\R^d,\\
\displaystyle&\displaystyle \rho^i(t_0,\cdot) = \mu_i,\  u_i(T,\cdot) = G(\cdot, \rho^i(T,\cdot)), & \R^d,
\end{array}
\right.
\end{equation}
where $\mu_i=\mathcal{L}_{\xi_i}$. To obtain our desired result, we need to work with $\pa_x u_i.$ By the regularity we assumed for $V$ in $x$, we cannot directly differentiate the backward SPDE in \eqref{eq:SPDEi} in $x$. Therefore we will mollify the function $u_i$. Let $\{\zeta_n\}_n$ be a sequence of densities in $C_c^\infty(B_{\frac{1}{n}}(0))$, as standard convolution kernels. Define
\begin{equation}\label{eq:uin}
u_{i,n}(t,x):=\int_{\mathbb R^d}u_i(t,x-y)\zeta_n(y)\dd y,\quad w_{i,n}(t,x):=\pa_{x}u_{i,n}(t,x)\quad\text{and}\quad w_{i}:=\pa_x u_i(t,x).
\end{equation}
Then
\begin{align}\label{nBSPDE}
 \dd u_{i,n}(t, x)&=  \int_{\mathbb R^d}\zeta_n(y)v_i(t,x-y)\dd y\cdot \dd B_t^0\\
 &- \Big[\frac{\beta^2}{2} \Delta u_{i,n}(t,x) +\int_{\mathbb R^d}\zeta_n(y) \big[\beta{\rm div}_x v_i(t,x-y) -H(x-y,\rho_t^i,w_i(t,x-y))\dd y\Big]\dd t. \nonumber
\end{align}
We then differentiate \eqref{nBSPDE} in $x$ and obtain
\begin{align}\label{paxBSPDE}
\dd w_{i,n}(t, x)&=\int_{\mathbb R^d}\zeta_n(x-y)\pa_xv_i(t,y)\dd y\cdot \dd B_t^0 - \Big[\frac{\beta^2}{2} \Delta w_{i,n}(t,x)\nonumber\\
 & - \int_{\mathbb R^d}\zeta_n(y)\big[ \pa_x H(x-y,\rho_t^i,w_i(t,x-y))+\pa_pH(x-y,\rho_t^i,w_i(t,x-y))\pa_xw_i(t,x-y)\big]\dd y\nonumber\\
 &+\beta\int_{\mathbb R^d}\pa_x\zeta_n(y){\rm div}_x v_i(t,x-y)\dd y\Big]\dd t.
\end{align}
By the It\^o--Wentzell formula, we have
\begin{align*}
\dd w_{i,n}(t,X_t^{i})&=\bigg\{\int_{\mathbb R^d}\zeta_n(y)\big[\pa_xH(X_t^{i}-y,\rho_t^i,w_i(t,X_t^i-y))\\
&+\pa_pH(X_t^i-y,\rho_t^i,w_i(t,X^i_t-y))\pa_xw_i(t,X_t^i-y)\big]\dd y-\pa_pH(X_t^i,\rho_t^i,w_i(t,X_t^i))\pa_{x}w_{i,n}(t,X_t^i)\bigg\}\dd t\\
&+\Big[\int_{\mathbb R^d}\zeta_n(X_t^i-y)\pa_xv_i(t,y)\dd y+\beta \pa_xw_{i,n}(t,X_t^i)\Big]\cdot\dd B_t^{0}.
\end{align*}
Letting $n\to+\infty$ in the above equation, we have
\begin{align}\label{eq:wi}
\dd w_i(t,X_t^{i})=\pa_xH(X_t^{i},\rho_t^i,w_i(t,X_t^i))\dd t+\left[\pa_xv_i(t,X_t^i)+\beta \pa_xw_i(t,X_t^i)\right]\cdot\dd B_t^{0}.
\end{align}
Now set $\bar X_t:=X_t^1-X_t^2$. Then by It\^o's formula we have
\begin{align}\label{eq:diffdis}
\dd\langle w_1(t,X_t^1)-w_2(t,X_t^2),\bar X_t\rangle&=\langle \bar X_t,\dd(w_1(t,X_t^1)-w_2(t,X_t^2))\rangle+\langle w_1(t,X_t^1)-w_2(t,X_t^2),\dd\bar X_t\rangle\nonumber\\
&=\Big[\langle \bar X_t,\pa_xH(X_t^1,\rho^1_t,w_1(t,X_t^1))-\pa_xH(X_t^2,\rho^2_t,w_2(t,X_t^2))\rangle\nonumber\\
&-\langle w_1(t,X_t^1)-w_2(t,X_t^2),\pa_pH(X_t^1,\rho^1_t,w_1(t,X_t^1))-\pa_pH(X_t^2,\rho^2_t,w_2(t,X_t^2))\rangle\Big]\dd t\nonumber\\
&+\Big[\bar X_t^{\top}\left(\pa_xv_1(t,X_t^1)-\pa_xv_2(t,X_t^2)+\beta(\pa_xw_1(t,X_t^1)-\pa_xw_2(t,X_t^2))\right)\Big]\cdot\dd B_t^0.
\end{align}
Integrating both sides from $t$ to $T$ and taking expectation, we obtain
\begin{align*}
\mathbb E\Big[\langle w_1(t,X_t^1)-w_2(t,X_t^2),\bar X_t\rangle\Big]&=\mathbb E\Big[\langle\pa_xG(X_T^1,\rho_T^1)-\pa_xG(X_T^2,\rho_T^2),\bar X_T\rangle\Big]\\
&-\int_t^T\mathbb E\Big[\langle \bar X_s,\pa_xH(X_s^1,\rho_s^1,w_1(s,X_s^1))-\pa_xH(X_s^2,\rho_s^2,w_2(s,X_s^2))\rangle\\
&-\langle w_1(s,X_s^1)-w_2(s,X_s^2),\pa_pH(X_s^1,\rho_s^1,w_1(s,X_s^1))-\pa_pH(X_s^2,\rho_s^2,w_2(s,X_s^2))\rangle\Big]\dd s\\
&\ge 0,
\end{align*}
where in the last line we have used the displacement monotonicity conditions of $G$ and $H$. By the definition of $w_i$, we have that $V(t,\cdot,\cdot)$ is displacement monotone for each $t\in [0,T]$.
\end{proof}

\begin{prop}\label{prop:Lip}
Suppose that $G,H$ satisfy Assumptions \ref{assum: G} and \ref{assum: H} and that $V$ is a classical solution to the master equation with bounded $\pa_{xx}V$ and $\pa_{x\mu}V$ on $[0,T]$. Suppose moreover, that $V(t,\cdot,\cdot)$ satisfies the displacement monotonicity condition for each $t\in [0,T]$. Then  $\pa_x V$ is Lipschitz continuous in $\mu$ with respect to the $W_2$ metric and the Lipschitz constant depends only on the data and $\|\pa_{xx}V\|_{L^\infty}$. 
\end{prop}
\begin{proof}
We first recall that $u_i(t,x):=V(t,x,\rho_t^i)$, $w_i(t,x)=\pa_x V(t,x,\rho_t^i)$ for $i=1,2$, and $\bar w:=w_1-w_2$, where $(u_i,\rho^i,v_i)$ ($i=1,2$) is the solution to the MFG system \eqref{eq:SPDEi}. Setting $\bar\xi:=\xi_1-\xi_2$, integrating both sides of \eqref{eq:diffdis} from $t_0$ to $t$ and taking expectation, we obtain by Young's inequality
\begin{align*}
\mathbb E\Big[\langle\pa_xV(t,X_t^1,\rho_t^1)&-\pa_xV(t,X_t^2,\rho_t^2),\bar X_t\rangle\Big]=\mathbb E\Big[\langle w_1(t_0,\xi_1)-w_2(t_0,\xi_2),\bar \xi\rangle\Big]\\
&+\int_{t_0}^t\mathbb E\Big[\langle \bar X_s,\pa_xH(X_s^1,\rho_s^1,w_1(s,X_s^2))-\pa_xH(X_s^2,\rho_s^2,w_2(s,X_s^2))\rangle\\
&-\langle w_1(s,X_s^1)-w_2(s,X_s^2),\pa_pH(X_s^1,\rho_s^1,w_1(s,X_s^1))-\pa_pH(X_s^2,\rho_s^2,w_2(s,X_s^2))\rangle\Big]\dd s\\
&\leq\mathbb E\Big[\langle w_1(t_0,\xi_1)-w_2(t_0,\xi_2),\bar \xi\rangle\Big]\\
&+\int_{t_0}^t\left\{C\mathbb E\big[|\bar X_s|^2\big]+C\mathbb E\big[|\bar X_s||w_1(s,X_s^1)-w_2(s,X_s^2)|\big]-c_0\mathbb E\big[|w_1(s,X_s^1)-w_2(s,X_s^2)|^2\big]\right\}\dd s\\
&\leq\mathbb E\Big[\langle w_1(t_0,\xi_1)-w_2(t_0,\xi_2),\bar \xi\rangle\Big]+\int_{t_0}^t\left\{C\mathbb E\big[|\bar X_s|^2\big]-\frac{c_0}{2}\mathbb E\big[|w_1(s,X_s^1)-w_2(s,X_s^2)|^2\big]\right\}\dd s,
\end{align*}
where in the penultimate inequality we have used the uniform Lipschitz continuity of $\partial_x H$ and $\partial_p H$, as well as the strong convexity of $H$ in the $p$-variable.

Then by the displacement monotonicity of $V(t,\cdot,\cdot)$ and the Lipschitz continuity of $w_1(t_0,\cdot)$, by rearranging the previous inequality and increasing the value of $C$, we obtain
\begin{align}\label{eq:wLip}
\int_{t_0}^t\mathbb E\big[|w_1(s,X_s^1)-w_2(s,X_s^2)|^2\big]\dd s&\leq C\mathbb E[|\bar\xi|^2]+ C\mathbb E\Big[\langle \bar w(t_0,\xi_2),\bar\xi\rangle\Big]+C\int_{t_0}^t\mathbb E\big[|\bar X_s|^2\big]\dd s\nonumber\\
&\leq C\mathbb E[|\bar\xi|^2]+ C\mathbb E\Big[ |\bar w(t_0,\xi_2)|^2\Big]^{\frac{1}{2}}\mathbb E[|\bar\xi|^2]^{\frac{1}{2}}+C\int_{t_0}^t\mathbb E\big[|\bar X_s|^2\big]\dd s
\end{align}
Derived from \eqref{eq:Xi}, we obtain by \eqref{eq:wLip}
\begin{align*}
\mathbb E[|\bar X_t|^2]&\leq C\mathbb E[|\bar\xi|^2]+C\int_{t_0}^t\mathbb E\big[|\pa_pH(X_s^1,\rho_s^1,w_1(s,X_s^1))-\pa_pH(X_s^2,\rho_s^2,w_2(s,X_s^2))|^2\big]\dd s\\
&\leq C\mathbb E[|\bar\xi|^2]+C\int_{t_0}^t\mathbb E\big[|\bar X_s|^2\big]+\mathbb E\big[|w_1(s,X_s^1)-w_2(s,X_s^2)|^2\big]\dd s\\
&\leq C\mathbb E[|\bar\xi|^2]+C\int_{t_0}^t\mathbb E\big[|\bar X_s|^2\big]\dd s+C\mathbb E\Big[|\bar w(t_0,\xi_2)|^2\Big]^{\frac{1}{2}}\mathbb E[|\bar\xi|^2]^{\frac{1}{2}}
\end{align*}
By Gr\"onwall's inequality, we have
\begin{align}\label{eq:Ep1p2}
\mathbb E[W_2^2(\rho_t^1,\rho_t^2)]\leq\mathbb E[|\bar X_t|^2]\leq C\mathbb E[|\bar\xi|^2]+C\mathbb E\Big[|\bar w(t_0,\xi_2)|^2\Big]^{\frac{1}{2}}\mathbb E[|\bar\xi|^2]^{\frac{1}{2}},\ \ \forall t\in [t_0,T].
\end{align}

For $(\rho^i_t)_{t\in[0,T]}$, $i=1,2$ given, with $\rho^i_{t_0}=\mu_i$, by the assumptions on $V$, we have that the following equations have strong solutions
$$
\displaystyle X_t^{\xi_i,x}=x-\int_{t_0}^t\pa_pH(X_s^{\xi_i,x},\rho^i_s,\pa_xV(t,X_t^{\xi_i,x},\rho^i_t))\dd s+\beta B_t^{0,t_0},
$$
and
\begin{equation}\label{eq:BSDEx}
\displaystyle Y_t^{\xi_i,x}=\partial_xG(X_T^{\xi_i,x},\rho^i_T)-\int_t^T \partial_xH(X_s^{\xi_i,x},\rho^i_s,Y_s^{\xi_i,x})\dd s-\beta\int_t^TZ_s^{0,\xi_i,x}\cdot \dd B_s^0.
\end{equation}
In particular, the solution to \eqref{eq:BSDEx} can be obtained via the representation formulas
\begin{align*}
Y_t^{\xi_i,x}:=\pa_xV(t,X_t^{\xi_i,x},\rho^i_t)\quad\text{and}\quad Z_t^{0,\xi_i,x}=\pa_{xx}V(t,X_t^{\xi_i,x},\rho^i_t)+\tilde\E_{\mathcal{F}_t}[\pa_{x\mu}V(t,X_t^{\xi_i,x},\rho^i_t,\tilde X_t^{\xi_i,x})],\quad t\in [t_0,T].
\end{align*}
By the global Lipschitz continuity of $\partial_p H$ and $\partial_x V(t,\cdot,\rho)$, standard results for SDEs and Gr\"onwall's inequality yield that

\begin{equation}\label{eq:DeltaXt0x}
\sup_{s\in [t_0,T]}\Big(\mathbb E\Big[\big|X_s^{\xi_1,x}-X_s^{\xi_2,x}\big|^2\Big]\Big)^{\frac{1}{2}}\leq C\int_{t_0}^T\mathbb E\big[\|\pa_xV(s,\cdot,\rho^1_s)-\pa_x V(s,\cdot,\rho^2_s)\|_{L^\infty(\mathbb R^d)}+W_2(\rho_s^1,\rho_s^2)\big]\dd s.
\end{equation}
We note that the first term on the right hand side of the previous inequality is finite because of the assumptions on $V$. Letting $t=t_0$ and taking expectation on \eqref{eq:BSDEx}, we have
\begin{align*}
\pa_xV(t_0,x,\rho^i_{t_0})=\mathbb E\left[\pa_x G(X_T^{\xi_i,x},\rho_T^i)\right]-\int_{t_0}^T\mathbb E\left[D_xH(X_s^{\xi_i,x},\rho_s^i,Y_s^{\xi_i,x})\right]\dd s,\quad \text{for $i=1,2$,}
\end{align*}
and thus by the Lipschitz continuity assumptions on $\pa_x G$ and $\pa_x H$ and by \eqref{eq:DeltaXt0x} we have
{\small
\begin{align*}
|\pa_xV(t_0,x,\mu^1)-\pa_x V(t_0,x,\mu^2)|&\le C\left[\sup_{s\in[t_0,T]}\mathbb EW_2(\rho^1_s,\rho^2_s)+\int_{t_0}^T\mathbb E\left[\left|\pa_xV(s,X_s^{\xi_1,x},\rho^1_s)-\pa_x V(s,X_s^{\xi_2,x},\rho^2_s)\right|\right]\dd s\right]\\
&+C\left[\mathbb E |X_T^{\xi_1,x}-X_T^{\xi_2,x}|+\int_{t_0}^T\mathbb E\left|X_s^{\xi_1,x}-X_s^{\xi_2,x}\right|\dd s\right]\\
&\le C\left[\sup_{s\in[t_0,T]}\mathbb EW_2(\rho^1_s,\rho^2_s)+\int_{t_0}^T\mathbb E\left[\left|\pa_xV(s,X_s^{\xi_1,x},\rho^1_s)-\pa_x V(s,X_s^{\xi_2,x},\rho^1_s)\right|\right]\dd s\right]\\
&+C\int_{t_0}^T\mathbb E\left[\left|\pa_xV(s,X_s^{\xi_2,x},\rho^1_s)-\pa_x V(s,X_s^{\xi_2,x},\rho^2_s)\right|\right]\dd s\\
&+C\left[\mathbb E |X_T^{\xi_1,x}-X_T^{\xi_2,x}|+\int_{t_0}^T\mathbb E\left|X_s^{\xi_1,x}-X_s^{\xi_2,x}\right|\dd s\right]\\
&\le C\left[\sup_{s\in[t_0,T]} \mathbb E W_2(\rho^1_s,\rho^2_s)+\int_{t_0}^T\mathbb E\left\|\pa_xV(s,\cdot,\rho^1_s)-\pa_x V(s,\cdot,\rho^2_s)\right\|_{L^\infty(\mathbb R^d)}\dd s\right]\\
&+C\left[\left(\mathbb E\big[ |X_T^{\xi_1,x}-X_T^{\xi_2,x}|^2\big]\right)^\frac12+\int_{t_0}^T\left(\mathbb E\big[|X_s^{\xi_1,x}-X_s^{\xi_2,x}|^2\big]\right)^\frac12\dd s\right]\\
&\le C\left[\sup_{s\in[t_0,T]}\mathbb EW_2(\rho^1_s,\rho^2_s)+\int_{t_0}^T\mathbb E\left\|\pa_xV(s,\cdot,\rho^1_s)-\pa_x V(s,\cdot,\rho^2_s)\right\|_{L^\infty(\mathbb R^d)}\dd s\right].
\end{align*}}

By Gronwall's inequality, we derive
\begin{align}\label{eq:Du}
\sup_{s\in[t_0,T]}\mathbb E\left\|\pa_xV(s,\cdot,\rho^1_s)-\pa_x V(s,\cdot,\rho^2_s)\right\|_{L^\infty(\mathbb R^d)}\le C\sup_{t\in[t_0,T]}\mathbb E W_2(\rho^1_t,\rho^2_t).
\end{align}
Plugging \eqref{eq:Du} into \eqref{eq:Ep1p2} and applying Young's inequality, we obtain
\begin{align*}
\sup_{t\in[t_0,T]} \mathbb EW_2(\rho_t^1,\rho_t^2)\leq C\left\{\mathbb E\left[\left|\xi_1-\xi_2\right|^2\right]\right\}^{\frac{1}{2}}.
\end{align*}
We can choose $\xi_i\in\mathbb L^2(\mathcal{F}_{t_0},\mu_i)$ to be such that $W_2(\mu_1,\mu_2)=\left\{\mathbb E\left[\left|\xi_1-\xi_2\right|^2\right]\right\}^{\frac{1}{2}}$ and thus
\begin{align}\label{eq:rhonew}
\sup_{t\in[t_0,T]}\mathbb E W_2(\rho_t^1,\rho_t^2)\leq C W_2(\mu_1,\mu_2).
\end{align}
Combining \eqref{eq:Du} and \eqref{eq:rhonew}
\begin{align*}
\sup_{s\in[t_0,T]}\mathbb E\left\|\pa_xV(s,\cdot,\rho^1_s)-\pa_x V(s,\cdot,\rho^2_s)\right\|_{L^\infty(\mathbb R^d)}\le CW_2(\mu_1,\mu_2).
\end{align*}
Thus, the thesis of this proposition follows.
\end{proof}

\begin{cor}\label{cor: space and W2}
Suppose that $V$ is a classical solution to the master equation. Then $\pa_x V$ is uniformly Lipschitz in space and measure variables with respect to the $W_2$ metric in the case of the measure component. Furthermore, the Lipschitz constant is universal. 
\end{cor}
\begin{proof}
From Lemma \ref{lem:dismon convex} and Proposition \ref{prop:monotone} we have that $\pa_{xx} V \geq 0$ and from Proposition \ref{cor:Vsemi-concave} we get that $\pa_{xx} V \leq CI_d$. Hence $\abs{\pa_{xx}V}$ is bounded by a universal constant and so $\pa_x V$ is Lipschitz continuous in space with universal Lipschitz constant. 

The Lipschitz continuity in the measure variable comes from Proposition \ref{prop:Lip}. 
\end{proof}}

\section{A Priori $W_1$-Lipschitz estimates on $\partial_x V(x,\cdot)$}\label{sec:4}

Several FBSDE systems will play a crucial role in our analysis.

\subsection{FBSDE of Pontryagin type} Let $t_0\in[0,T)$ and ${\xi\in\mathbb{L}^2(\mathcal{F}_{t_0})}$ and on the time interval $[t_0,T]$ we consider 
\begin{equation}\label{FBSDE1} 
\left\{
\begin{array}{lcl}
{X_t^{\xi}} &=&\displaystyle \xi  -\int_{{t_0}}^t \partial_p H(X_s^{\xi}, \rho_s, Y_s^{\xi}) \dd s + {\beta}B_t^{{0,t_0}}\\
{Y_t^{\xi}} &=&\displaystyle \partial_x G(X_T^{\xi}, \rho_T) {-}\int_t^T\partial_x H(X_s^{\xi}, \rho_s, Y_s^{\xi}) \dd s - \int_t^T Z_s^{0,\xi} \dd B_s^{{0}}
\end{array}
\right.
\end{equation}
where {$\rho_{t_0}:=\mathcal{L}_\xi$} and $\rho_t:= \mathcal{L}_{X_t^{\xi}| \mathcal F_t^0}$.

\begin{lem}\label{lem:representation}
Suppose that $G,H$ satisfy Assumptions \ref{assum: G} and \ref{assum: H} and that $V$ is a classical solution to the master equation with bounded $\pa_{xx}V$ and $\pa_{x\mu}V$ on $[0,T]$. Then 
we have the representation formulas $Y_t^\xi = \pa_x V(t, X_t^\xi, \rho_t) $ and {$Z_t^{0,\xi}=\beta\left(\pa_{xx}V(t,X_t^\xi,\rho_t)+\tilde\E_{\mathcal{F}_t}[\pa_{x\mu}V(t,X_t^\xi,\rho_t,\tilde X_t^\xi)]\right)$}.
\end{lem}

\begin{proof}
This result is well-known for experts and its proof follows the same lines as the proofs of \cite[Proposition 5.42]{Carmona2018}, \cite[Remark 57]{ChaCriDel22}, \cite[Theorem 6.3]{FourAuthor} and \cite[Theorem 4.1]{MesMou}.
\end{proof}

\medskip

We also consider the standard system
\begin{equation}\label{FBSDE2}
\left\{
\begin{array}{lcl}
{X_t^{x}} &=& x  + {\beta}B_t^{{0,t_0}} \\
{Y_t^{x,\xi}} &=& \displaystyle\partial_x G(X_T^{x}, \rho_T) {-}\int_t^T\partial_x H(X_s^{x}, \rho_s, Y_s^{x,\xi}) \dd s - \int_t^T Z_s^{0,x,\xi} \dd B_s^{{0}}
\end{array}
\right.
\end{equation}
and the alternative system
\begin{equation}\label{FBSDE3}
\left\{ 
\begin{array}{lcl}
{X_t^{\xi,x}} &=&\displaystyle {x}  -\int_{{t_0}}^t \partial_p H(X_s^{\xi,x}, \rho_s, Y_s^{\xi,x}) \dd s + \beta B_t^{{0,t_0}}\\
{Y_t^{\xi,x}} &=&\displaystyle \partial_x G(X_T^{\xi,x}, \rho_T) {-} \int_t^T\partial_x H(X_s^{\xi,x}, \rho_s, Y_s^{\xi,x}) \dd s - \int_t^T Z_s^{0,\xi,x} \dd B_s^{{0}}, 
\end{array}
\right.
\end{equation}
which also have the corresponding representation formulas. {Note the the difference between the variables $({Y_t^{x,\xi}}, Z_s^{0,x,\xi})$ and $({Y_t^{\xi,x}},Z_s^{0,\xi,x})$ which is expressed in the superscript labels. We underline that the solutions of \eqref{FBSDE2} and \eqref{FBSDE3} depend implicitly on $\xi$ via the flow of measures $(\rho_t)_{t\in[t_0,T]}$.}

\indent We emphasize that these differ from the systems in \cite{FourAuthor} in that for us the variable $Y$ plays the role of the momentum along the characteristics whereas in \cite{FourAuthor} it is the value function along the characteristics. {All the previous FBSDE systems presented above are strongly linked to the Pontryagin system associated to the stochastic maximum principle (or to the classical Hamiltonian system when $\beta=0$), and the systems used in \cite{FourAuthor} have no such connection.

\indent Both previously presented systems have similar representation formulas as in Lemma \ref{lem:representation}. Indeed, 
$$
Y_t^{x,\xi} = \pa_x V(t, X_t^{x}, \rho_t),\ \ Y_t^{\xi,x} = \pa_x V(t, X_t^{\xi,x}, \rho_t).
$$
These are clearly different quantities, as in particular, if $\beta=0$, we simply have $Y_t^{x,\xi} = \pa_x V(t,x, \rho_t)\neq \pa_x V(t,X^{\xi,x}_t,\rho_t)$, except when $t=t_0$, when $Y^{x,\xi}_{t_0} = Y^{\xi,x}_{t_0} = \pa_x V(t_0,x,\rho_{t_0})$.
}

\subsection{Intuition}
To help to give some intuition for the role of the different systems considered above, let us consider $\beta = 0$ for this subsection. In the deterministic case we can use $X^{\xi, x}_t$ to define $\rho_t$. In particular we get that $\rho_t = \mathcal L(X^{\xi, \xi(\cdot)}_t) = \mathcal L(X^{\xi, \cdot}_t \circ \xi)$ (note that $X^{\xi, x}_{t_0} = x$). 

Our objective is to develop some equations that give a representation for $\pa_{\mu x}V$. Since $Y_t^{x,\xi} = \pa_x V(t, x, \rho_t)$ (since $\beta = 0$) it would be natural to try to differentiate the defining equation of $Y_t^{x,\xi}$ with respect to $\xi$ (this will become \eqref{tdYmu} below). Let us formally attempt this. Our equation is
\[
{Y_t^{x,\xi}} = \displaystyle\partial_x G(X_T^{x}, \rho_T) {-}\int_t^T\partial_x H(X_s^{x}, \rho_s, Y_s^{x,\xi}) \dd s
\]
and so we see that this comes down to differentiating the flow of measures $\mathcal L(X^{\xi, \cdot}_t \circ \xi)$ with respect to $\xi$. 
Replacing $\xi$ with $\xi + \eps e_1$ we get
\[
X_s^{\xi+\eps e_1, \cdot} ((\xi(\omega)+\eps e_1))
&\approx X_s^{\xi+\eps e_1, \cdot}(\xi(\omega)) + \eps \nabla X_s^{\xi, \cdot}(\xi(\omega)) \cdot e_1 \\
&\approx X_s^{\xi, \cdot}(\xi(\omega)) + \eps \nabla X_s^{\xi, \cdot}(\xi(\omega)) \cdot e_1 + \eps \delta X_s^{\xi, \cdot}(\xi(\omega))
\]
where $\delta X_s^{\xi, \cdot}$ represents the value that one gets when from $X^{\xi, \cdot}_s$ when perturbing $\xi$. So we see that an equation that gives the variation of ${Y_t^{x,\xi}}$ with respect to $\xi$ will involve two types of variations of $X$. The first is a gradient in space and the second is a variation with respect to $\xi$. Each of these will require their own system of FBSDEs which gives us three systems in total.

This also helps us understand why we need to consider the three systems above \eqref{FBSDE1},\eqref{FBSDE2},\eqref{FBSDE3}. Having \eqref{FBSDE2} is a matter of convenience as it provides the simplest representation formula. \eqref{FBSDE3} is necessary because we need to understand the gradient in space of $X$. In the case of no noise these two alone would have been sufficient. However in the presence of noise we must also consider \eqref{FBSDE1} because we cannot extract the $\rho_s$ directly from \eqref{FBSDE3}.

\subsection{FBSDEs for pointwise representation}
{
In order to gain the necessary a priori regularity estimates on $\pa_xV$ (notably the fact that it is $W_1$--Lipschitz continuous in the measure variable), we work at the level of linearized FBSDE systems. These are derived from \eqref{FBSDE1}, \eqref{FBSDE2} and \eqref{FBSDE3}. Linearization techniques combined with finite dimensional projections (in the measure variable) are underneath essentially all well-posedness results on master equations. This is typically carried out either at the PDE level using the MFG system (as for instance in \cite{CarDelLasLio, AmbMes}, etc.) or at the level of the Hamiltonian/FBSDE system (as for instance in \cite{Carmona2018,GanSwi,MouZha23,TwoAuthor2022,FourAuthor,MouZha22}, etc.)

\indent Consider $\{e_1,\dots,e_d\}\subset\R^d$ the canonical basis and for $k\in\{1,\dots,d\}$. First, we differentiate \eqref{FBSDE3} in the $e_k$ direction to obtain}
\begin{equation}
\label{tdYx}
\begin{cases}
	\displaystyle \nabla_{k} X_t^{\xi,x}&=\displaystyle e_k {-}\int_{t_0}^t \left\{(\nabla_k X_s^{\xi,x})^\top\partial_{xp}H(X_s^{\xi,x},\rho_s,Y_s^{\xi,x})
	+ (\nabla_k Y_s^{\xi,x})^\top\partial_{pp}H(X_s^{\xi,x},\rho_s,Y_s^{\xi,x})\right\}\dd s \\[5pt]
	\displaystyle \nabla_{k} Y_t^{\xi,x}&= \displaystyle \partial_{xx} G(X^{\xi,x}_T,\rho_T) \cdot \nabla_{k} X^{\xi,x}_T\\[5pt]
	\displaystyle\ &{-} \displaystyle \int_t^T \left\{\partial_{xx}H(X_s^{\xi,x},\rho_s,Y_s^{\xi,x})\cdot\nabla_k X_s^{\xi,x} {+} \partial_{px}H(X_s^{\xi,x},\rho_s,Y_s^{\xi,x})\cdot\nabla_k Y_s^{\xi,x}\right\}\dd s\\
	&\displaystyle {-}  \int_t^T \nabla_{k} Z_s^{0,\xi,x}\cdot \dd B_s^{0}.
\end{cases}
\end{equation}

\begin{eqnarray}
\label{tdYx-}
\left\{\begin{array}{ll}
\displaystyle \nabla_{k} \mathcal X_t^{\xi,x}&=\displaystyle{-} \int_{t_0}^t\Big\{(\nabla_k \mathcal X_s^{\xi,x})^{\top}\partial_{xp}H(X_s^{\xi},\rho_s,Y_s^\xi)+(\nabla_k \mathcal Y_s^{\xi,x})^{\top}\partial_{pp}H(X_s^{\xi},\rho_s,Y_s^\xi)\\[8pt]
\displaystyle & -\tilde{\mathbb E}_{\mathcal{F}_s}\left[(\nabla_{k} \tilde X_s^{\xi,x})^{\top}(\partial_{\mu p} H)(X^{\xi}_s,\rho_s,\tilde X_s^{\xi,x}, Y^{\xi}_s)
\displaystyle +(\nabla_{k} \tilde {\mathcal X}_s^{\xi,x})^{\top}\partial_{\mu p} H(X^{\xi}_s,\rho_s,\tilde X_s^{\xi},Y^{\xi}_s)\right] \Big\}\dd s \\[5pt]
\displaystyle \nabla_{k} \mathcal Y_t^{\xi,x}&= \displaystyle \partial_{xx} G(X^{\xi}_T,\rho_T) \cdot \nabla_{k} \mathcal X^{\xi,x}_T\\[5pt]
&+\displaystyle\tilde{\mathbb E}_{\mathcal{F}_T}\big[\partial_{\mu x}G(X_T^\xi,\rho_T,\tilde X_T^{\xi,x})\cdot \nabla_{k} \tilde X_T^{\xi,x}+\partial_{\mu x}G(X_T^\xi,\rho_T, \tilde X_T^\xi)\cdot\nabla_{k}\tilde {\mathcal X}_T^{\xi,x}\big]\\[5pt]
&\displaystyle {-} \int_t^T\Big\{ \partial_{xx} H\big(X_s^{\xi},\rho_s,Y^{\xi}_s)\cdot \nabla_{k} \mathcal X^{\xi,x}_s+  \partial_{px} H\big(X_s^{\xi},\rho_s,Y^{\xi}_s)\cdot \nabla_{k} \mathcal Y^{\xi,x}_s\\[5pt]
&\displaystyle +\tilde{\mathbb E}_{\mathcal{F}_s}\big[\partial_{\mu x}H\big(X^{\xi}_s,\rho_s,\tilde X_s^{\xi,x}, Y^{\xi}_s )\cdot\nabla_{k} \tilde X_s^{\xi,x}+\partial_{\mu x}H\big(X^{\xi}_s, \rho_s,\tilde X_s^{\xi},Y^{\xi}_s)\cdot\nabla_{k} \tilde {\mathcal X}_s^{\xi,x}\big]\Big\}\dd s\\[5pt]
&\displaystyle {-} \int_t^T \nabla_{k} \mathcal Z_s^{0,\xi,x}\cdot \dd B_s^{0},
\end{array}\right.
\end{eqnarray}
and 
\begin{align}\label{tdYmu}
\nabla_{\mu_k} Y_t^{x,\xi,\tilde x}&=\tilde{\mathbb E}_{\mathcal{F}_T}\big[\partial_{\mu x} G (X_T^{x},\rho_T,\tilde X_T^{\xi,\tilde x})\cdot \nabla_{k} \tilde X_T^{\xi,\tilde x}+\partial_{\mu x} G (X_T^{x},\rho_T,\tilde X_T^\xi)\cdot \nabla_{k}\tilde {\mathcal X}_T^{\xi,\tilde x}\big]\nonumber\\
&\displaystyle {-} \int_t^T \Big\{\partial_{px}H(X_s^x,\rho_s,Y_s^{x,\xi})\cdot \nabla_{\mu_k} Y_s^{x,\xi,\tilde x}\\
&\displaystyle +\tilde{\mathbb E}_{\mathcal{F}_s}\big[\partial_{\mu x} H (X_s^{x},\rho_s,\tilde X_s^{\xi,\tilde x},Y_s^{x,\xi})\cdot\nabla_{k} \tilde X_s^{\xi,\tilde x}+\partial_{\mu x} H (X_s^{x},\rho_s,\tilde X_s^\xi,Y_s^{x,\xi})\cdot \nabla_{k}\tilde {\mathcal X}_s^{\xi,\tilde x}\big]\Big\}\dd s\nonumber\\
&\displaystyle{-} \int_t^T \nabla_{\mu_k} Z_s^{0,x,\xi,\tilde x}\cdot \dd B_s^{{0}}.\nonumber
\end{align}

\begin{rmk}
The motivation behind the notation $\nabla_{\mu_k}$ in \eqref{tdYmu} is that this is linked to the $k^{th}$ component for the Wasserstein gradient. Later we will also use the notation $\pa_{\mu_k}$ with a similar purpose, i.e. $\pa_{\mu_k}F:=\pa_\mu F\cdot e_k$, for any $F$ regular enough.
\end{rmk}

\begin{lem}\label{lem:delta}
Suppose that $G,H$ satisfy Assumptions \ref{assum: G} and \ref{assum: H}. There is a constant $\delta>0$ so that whenever $T - t_0 < \delta$ the systems \eqref{FBSDE1}, \eqref{tdYx}, and \eqref{tdYx-} have a unique solution, where $\delta$ depends only on $L^H$ and $L_2^G$ (these are the bounds on the second derivatives of $H$ and the bounds on Lipschitz constant of $\pa_x G$, with respect to space and $W_2$ in measure). 
Furthermore the solutions to these systems are bounded by controlled quantities, specifically there is a constant $C$ depending only on $T$, $L^H$, and $L_2^G$ so that if $A_t$ is one of $\nabla_{k} X_t^{\xi,x}, \nabla_{k} Y_t^{\xi,x}, \nabla_{k} \mathcal X_t^{\xi,x}$, or $\nabla_{k} \mathcal Y_t^{\xi,x}$ then
\[
\E\left[\sup_{s \in [t_0, T]} \abs{A_s}^2\right] \leq C
\]
\end{lem}
It is crucial in the above lemma that the constant $\delta$ depends only on the $W_2$-Lipschitz constant of $\pa_x G$ and not on the $W_1$-Lipschitz constant.

\begin{proof}
The proof is similar to  \cite[Section 9]{MouZha23} and \cite[Proposition 6.2(i)]{FourAuthor}.

\indent The described short time existence and uniqueness for \eqref{FBSDE1} follows from \cite[Theorem 5.4]{Carmona2018}. 

For \eqref{tdYx} we use \cite[Theorem 8.2.1]{Zhang2017}. In particular we note that because \eqref{tdYx} is linear (in $\nabla_{k} X^{\xi,x}_T$ and $\nabla_{k} Y^{\xi,x}_T$) the short time interval only depends on the absolute value of the coefficients which only include $\pa_{xx} G$ and second derivatives of $H$. Finally from the cited theorem we see that $\E\left[\sup_{s \in [t_0, T]} \abs{A_s}^2\right]$ is bounded for $A_s = \nabla_{k} X_t^{\xi,x}$ or $\nabla_{k} Y_t^{\xi,x}$.

Next we consider \eqref{tdYx-}. Again this is a linear FBSDE. The proof follows very similiarly to \cite[Theorem 8.2.1]{Zhang2017} however a modification is required.

Consider the standard mapping $F$ given by $y_s$ maps to $\nabla_{k} \mathcal Y_s^{\xi,x}$ where $\nabla_{k} \mathcal Y_s^{\xi,x}$ is the solution to 
\begin{eqnarray*}
\left\{\begin{array}{ll}
\displaystyle \nabla_{k} \mathcal X_t^{\xi,x}&=-\displaystyle \int_{t_0}^t\Big\{(\nabla_{k} \mathcal X_s^{\xi,x})^{\top}\partial_{xp}H(X_s^{\xi},\rho_s,Y_s^\xi)+(y_s)^{\top}\partial_{pp}H(X_s^{\xi},\rho_s,Y_s^\xi)\\[8pt]
\displaystyle & +\tilde{\mathbb E}_{\mathcal{F}_s}\left[(\nabla_{k} \tilde X_s^{\xi,x})^{\top}(\partial_{\mu p} H)(X^{\xi}_s,\rho_s,\tilde X_s^{\xi,x}, Y^{\xi}_s)
\displaystyle +(\nabla_{k} \tilde {\mathcal X}_s^{\xi,x})^{\top}\partial_{\mu p} H(X^{\xi}_s,\rho_s,\tilde X_s^{\xi},Y^{\xi}_s)\right] \Big\}\dd s \\[5pt]
\displaystyle \nabla_{k} \mathcal Y_t^{\xi,x}&= \displaystyle \partial_{xx} G(X^{\xi}_T,\rho_T) \cdot \nabla_{k} \mathcal X^{\xi,x}_T\\[5pt]
&+\displaystyle\tilde{\mathbb E}_{\mathcal{F}_T}\big[\partial_{\mu x}G(X_T^\xi,\rho_T,\tilde X_T^{\xi,x})\cdot \nabla_{k} \tilde X_T^{\xi,x}+\partial_{\mu x}G(X_T^\xi,\rho_T, \tilde X_T^\xi)\cdot\nabla_{k}\tilde {\mathcal X}_T^{\xi,x}\big]\\[5pt]
&\displaystyle {-} \int_t^T\Big\{ \partial_{xx} H\big(X_s^{\xi},\rho_s,Y^{\xi}_s)\cdot \nabla_{k} \mathcal X^{\xi,x}_s+  \partial_{px} H\big(X_s^{\xi},\rho_s,Y^{\xi}_s)\cdot y_s\\[5pt]
&\displaystyle +\tilde{\mathbb E}_{\mathcal{F}_s}\big[\partial_{\mu x}H\big(X^{\xi}_s,\rho_s,\tilde X_s^{\xi,x}, Y^{\xi}_s )\cdot\nabla_{k} \tilde X_s^{\xi,x}+\partial_{\mu x}H\big(X^{\xi}_s, \rho_s,\tilde X_s^{\xi},Y^{\xi}_s)\cdot\nabla_{k} \tilde {\mathcal X}_s^{\xi,x}\big]\Big\}\dd s\\[5pt]
&\displaystyle {-} \int_t^T \nabla_{k} \mathcal Z_s^{0,\xi,x}\cdot \dd B_s^{0},
\end{array}\right.
\end{eqnarray*}
\newcommand{\mc}{\nabla_k \mathcal }
 We will show that $F$ is a contraction mapping under the norm given by $\norm{y_s}^2 = \sup_s \E(\abs{y_s}^2)$ when $T$ is sufficiently small (for now assume $T < 1$). Indeed fix some $y^1, y^2$ denote by $\Delta y := y^1 - y^2$ and $\mc X^i$ be the solutions to the above system with $y_s = y_s^i$. Let $\Delta X := \mc X^1-\mc X^2$ and $\Delta Y := F(y^1) - F(y^2)$. Applying Gr\"onwall's inequality to the first equation in the system we see that $\Delta X$ satisfies $\norm{\Delta X} \leq C T \norm{\Delta y}$ where $C$ depends only on $L^H$. From the second equation we see that $\Delta Y$ satisfies the system 
 \begin{eqnarray*}
	\left\{\begin{array}{ll}
	\displaystyle \Delta Y_t &= \displaystyle \partial_{xx} G(X^{\xi}_T,\rho_T) \cdot  \Delta X_T\\[5pt]
	&+\displaystyle\tilde{\mathbb E}_{\mathcal{F}_T}\big[\partial_{\mu x}G(X_T^\xi,\rho_T, \tilde X_T^\xi)\cdot\tilde {\Delta X}_T\big]\\[5pt]
	&\displaystyle {-} \int_t^T\Big\{ \partial_{xx} H\big(X_s^{\xi},\rho_s,Y^{\xi}_s)\cdot  \Delta X^{\xi,x}_s+  \partial_{px} H\big(X_s^{\xi},\rho_s,Y^{\xi}_s)\cdot \Delta y_s\\[5pt]
	&\displaystyle +\tilde{\mathbb E}_{\mathcal{F}_s}\big[\partial_{\mu x}H\big(X^{\xi}_s, \rho_s,\tilde X_s^{\xi},Y^{\xi}_s)\cdot \tilde {\Delta X}_s^{\xi,x}\big]\Big\}\dd s\\[5pt]
	&\displaystyle {-} \int_t^T  \Delta \mathcal Z_s^{0,\xi,x}\cdot \dd B_s^{0},
	\end{array}\right.
	\end{eqnarray*}
The only term that may seem concerning is $\displaystyle\tilde{\mathbb E}_{\mathcal{F}_T}\big[\partial_{\mu x}G(X_T^\xi,\rho_T, \tilde X_T^\xi)\cdot\tilde {\Delta X}_T\big]$ since we want to claim that the short time interval depends only on the $W_2$-Lipschitz constant of $\pa_x G$ and not on the $W_1$-Lipschitz constant of $G$. However, note that 
\begin{equation*}
\displaystyle \\\abs{\tilde{\mathbb E}_{\mathcal{F}_T}\big[\partial_{\mu x}G(X_T^\xi,\rho_T, \tilde X_T^\xi)\cdot\tilde {\Delta X}_T\big]}	
\leq L_2^G \mathbb E(\Delta X_T^2)^{\frac12}
\end{equation*}
which follows from \eqref{WassersteinDerivativeDef}. It now follows from standard BSDE estimates (see \cite[Theorem 4.2.1]{Zhang2017}) that $\norm{\Delta Y} \leq C_1 T \norm{\Delta y} + C_2 \norm{\Delta X}$ where $C_1, C_2$ depend only on $L^H$ and $L_2^G$. Combining this with our estimate $\norm{\Delta X} \leq C T \norm{\Delta y}$ we see that $F$ is a contraction mapping as long as $T$ is sufficiently small and so there exists a unique solution to the above system. In particular the small time interval and corresponding bound will depend only on $L_2^G$ and not $\|\pa_{x\mu}G\|_{L^\infty}$. 

\end{proof}

\begin{cor}\label{cor: muk Y bound}
Suppose that $G,H$ satisfy Assumptions \ref{assum: G} and \ref{assum: H} and $\delta$ is given in Lemma \ref{lem:delta}. Then $\E \left[\abs{\nabla_{\mu_k} Y_{t_0}^{x,\xi,\tilde x}}^2\right]$ is bounded by a universal constant on $[t_0,T]$, for any $t_0$ with $T-t_0<\delta$. 
\end{cor}

\begin{proof}
We see that $\nabla_{\mu_k} Y_{t_0}^{x,\xi,\tilde x}$ is the solution to a linear BSDE with coefficients that are bounded by universal constants. 
\end{proof}

\subsection{Proof of the representation formula}
The proof is broken into four steps. In the first step we develop a system of FBSDE that gives a representation for $\mathbb{E}\big[\partial_{\mu x} V({t_0},x, \mu, \xi) \eta\big]$ where $\eta\in \mathbb{L}^2({\cal F}_{{t_0}})$ is arbitrary. In the next two steps we prove the representation formula for discrete and absolutely continuous measures respectively. Finally we prove it for general measures. Let us remark that the approach via passing through measures with discrete supports is a natural and typical one in the context of such representation formulas. We refer to the proof of \cite[Proposition 5.55]{Carmona2018} and to the discussion on pages 218-219 from \cite{CD1} on this matter.

\begin{prop}\label{prop:repr}
Assume that $G,H$ satisfy Assumptions \ref{assum: G} and \ref{assum: H} and $\delta$ is given in Lemma \ref{lem:delta}. Recall the system \eqref{FBSDE2} and define $\vec{U}(t_0,x,\mu):=Y_{t_0}^{x,\xi}$ for any $t_0$ with $T-t_0<\delta$. Then

\begin{equation}\label{pamuV}
	\partial_{\mu_k } \vec{U}(t_0,x,\mu,\tilde x)=\nabla_{\mu_k}Y_{t_0}^{x,\xi,\tilde x}
\end{equation}
and, moreover, $\pa_{\mu} \vec{U}$ is uniformly bounded by a universal constant.
\end{prop}

\begin{proof}
{\it Step 1.} For any  $\xi\in \mathbb{L}^2({\cal F}_{{t_0}}, \mu)$ and any scalar random variable $\eta\in \mathbb{L}^2({\cal F}_{{t_0}}, \mathbb{R})$, following standard arguments and by  the stability property of the involved systems we have
\begin{eqnarray}
	\label{tdXconv}
	\lim_{\varepsilon\to 0} \mathbb{E}\left[\sup_{{t_0}\le t\le T} \Big| {1\over \varepsilon}\big[X^{\xi + \varepsilon \eta e_1}_t - X^\xi_t\big] - \delta X^{\xi, \eta e_1}_t\Big|^2\right] = 0,
\end{eqnarray}
where $\left(\delta  X^{\xi, \eta e_1}, \delta  Y^{\xi, \eta e_1}, \delta  Z^{0,\xi, \eta e_1}\right)$ satisfies the linear McKean--Vlasov FBSDE
\begin{eqnarray}
	\label{tdmuFBSDE}
	\left\{\begin{array}{ll}
		\displaystyle  \delta  X^{\xi, \eta e_1}_t&\displaystyle=\eta e_1 {-} \int_{{t_0}}^t \Big\{ ( \delta  X_s^{\xi,\eta e_1})^{\top}\partial_{xp} H\big(X^{\xi}_s,\rho_s,Y^{\xi}_s)+( \delta  Y_s^{\xi,\eta e_1})^{\top}\partial_{pp} H\big(X^{\xi}_s,\rho_s,Y^{\xi}_s) \\ &+ \tilde {\mathbb{E}}_{{\cal F}_s}\big[\partial_{\mu p} H(X_s^{\xi},\rho_s, \tilde X^\xi_s,Y_s^{\xi})\cdot\delta  \tilde X^{\xi, \eta e_1}_{{s}}\big]\Big\}\dd s \\ [7pt]
		\displaystyle \delta  Y_t^{\xi, \eta e_1}&\displaystyle= \partial_{xx} G(X_T^{\xi},\rho_T) \cdot \delta  X^{\xi, \eta e_1}_T + \tilde {\mathbb{E}}_{{\cal F}_T}\big[\partial_{\mu x} G(X_T^{\xi},\rho_T, \tilde X^\xi_T)\cdot \delta  \tilde X^{\xi, \eta e_1}_T\big]\\
		\displaystyle & \displaystyle{-}\int_t^T\Big\{  \partial_{x x} H\big(X^{\xi}_s, \rho_s,Y^{\xi}_s)\cdot \delta  X^{\xi,\eta e_1}_s+   \partial_{px} H\big(X^{\xi}_s,\rho_s, Y^{\xi}_s)\cdot \delta  Y^{\xi,\eta e_1}_s\\
		\displaystyle &\displaystyle+\tilde {\mathbb{E}}_{{\cal F}_s}\big[\partial_{\mu x} H(X_s^{\xi},\rho_s, \tilde X^\xi_s,Y_s^{\xi})\cdot\delta  \tilde X^{\xi, \eta e_1}_{{s}}\big]\Big\}\dd s {-} \int_t^T\delta  Z_s^{0,\xi,\eta e_1}\cdot \dd B_s^{0}.
	\end{array}\right.
\end{eqnarray}

Specifically let $\delta  \Phi_t^{\xi, \eta e_1, \varepsilon} = \frac1\varepsilon(\Phi^{\xi + \varepsilon \eta e_1}_t - \Phi^\xi_t)$ for $\Phi \in \{X, Y, Z^0\}$. By substituting and subtracting in \eqref{FBSDE1} we see
\[
\delta X_t^{\xi, \eta e_1, \varepsilon} 
&= \eta e_1  - \frac1\varepsilon\int_{{t_0}}^t \Big\{\partial_p H(X_s^{\xi + \varepsilon \eta e_1}, \rho_s^{\xi + \varepsilon \eta e_1}, Y_s^{\xi + \varepsilon \eta e_1}) - \partial_p H(X_s^{\xi }, \rho_s^{\xi }, Y_s^{\xi})\Big\} \dd s\\
&= \eta e_1 {-} \int_{{t_0}}^t \Big\{ ( \delta  X_s^{\xi, \eta e_1, \varepsilon})^{\top}\partial_{xp} H\big(X^{\xi}_s,\rho_s,Y^{\xi}_s)+( \delta  Y_s^{\xi, \eta e_1, \varepsilon})^{\top}\partial_{pp} H\big(X^{\xi}_s,\rho_s,Y^{\xi}_s) \\ &+ \tilde {\mathbb{E}}_{{\cal F}_s}\big[\partial_{\mu p} H(X_s^{\xi},\rho_s, \tilde X^\xi_s,Y_s^{\xi})\cdot\delta  \tilde X^{\xi, \eta e_1, \varepsilon}_{{s}}\big]\Big\}\dd s + O(\varepsilon) 
\]
and
\[
\delta Y_t^{\xi, \eta e_1, \varepsilon} 
&=\partial_{xx} G(X_T^{\xi},\rho_T) \cdot \delta  X^{\xi, \eta e_1, \varepsilon}_T + \tilde {\mathbb{E}}_{{\cal F}_T}\big[\partial_{\mu x} G(X_T^{\xi},\rho_T, \tilde X^\xi_T)\cdot \delta  \tilde X^{\xi, \eta e_1, \varepsilon}_T\big]\\
\displaystyle & \displaystyle {-} \int_t^T\Big\{  \partial_{x x} H\big(X^{\xi}_s, \rho_s,Y^{\xi}_s)\cdot \delta  X^{\xi,\eta e_1, \varepsilon}_s+   \partial_{px} H\big(X^{\xi}_s,\rho_s, Y^{\xi}_s)\cdot \delta  Y^{\xi,\eta e_1, \varepsilon}_s\\
\displaystyle &\displaystyle+\tilde {\mathbb{E}}_{{\cal F}_s}\big[\partial_{\mu x} H(X_s^{\xi},\rho_s, \tilde X^\xi_s,Y_s^{\xi})\cdot\delta  \tilde X^{\xi, \eta e_1, \varepsilon}_T\big]\Big\}\dd s {-} \int_t^T\delta  Z_s^{0,\xi,\eta e_1, \varepsilon}\cdot \dd B_s^0 + O(\varepsilon).
\]
Note that aside from the $O(\varepsilon)$ term, $\left(\delta X_t^{\xi, \eta e_1, \varepsilon}, \delta Y_t^{\xi, \eta e_1, \varepsilon},{\delta  Z^{0,\xi, \eta e_1,\varepsilon}}\right)$ satisfies the exact same FBSDE system as $\left(\delta X_t^{\xi, \eta e_1,}, \delta Y_t^{\xi, \eta e_1},{\delta  Z^{0,\xi, \eta e_1}}\right).$ By the stability of FBSDE we get
\[
\lim_{\varepsilon\to 0} \mathbb{E}\left[\sup_{{t_0}\le t\le T} \Big| \delta X^{\xi, \eta e_1, \varepsilon}_t - \delta X^{\xi, \eta e_1}_t\Big|^2\right] = 0
\]
as desired.

Similarly to \eqref{tdXconv}, {using \eqref{FBSDE2}}, one can show that
\begin{eqnarray}
\label{tdYconv}
\lim_{\varepsilon\to 0} \mathbb{E}\left[\sup_{{t_0}\le t\le T} \Big| {1\over \varepsilon}\big[Y^{x, \xi + \varepsilon \eta e_1}_t - Y^{x,\xi}_t\big] - \delta  Y^{x, \xi, \eta e_1}_t\Big|^2\right] = 0,
\end{eqnarray}
where $\Big(\delta  Y^{x,\xi, \eta e_1}, \delta  Z^{0,x, \xi, \eta e_1}\Big)$ satisfies the linear (standard) BSDE
\begin{eqnarray}\label{eq:nabla mu X}
\left.\begin{array}{c}
\displaystyle \delta  Y_t^{x, \xi, \eta e_1}= \displaystyle  \tilde {\mathbb{E}}_{{\cal F}^0_T}\big[\partial_{\mu x} G(X_T^{x},\rho_T, \tilde X^\xi_T) \cdot\delta  \tilde X^{\xi, \eta e_1}_T\big] {-}\int_t^T\delta  Z_s^{0,x,\xi, \eta e_1}\cdot \dd B_s^0\\[7pt]
\displaystyle {-}\int_t^T \Big\{\partial_{p x} H(X_s^{x},\rho_s,Y_s^{x,\xi})\cdot \delta  Y^{x,\xi,\eta e_1}_s+\tilde {\mathbb{E}}_{{\cal F}_s}\big[\partial_{\mu x} H(X_s^{x},\rho_s,\tilde X_s^{\xi},Y_s^{x,\xi})\cdot \delta  \tilde X_s^{\xi,\eta e_1}\big]\Big\}\dd s 
\end{array}\right.
\end{eqnarray}
In particular, \eqref{tdYconv} implies,
\begin{eqnarray}
	\label{pamuVconv}
	\lim_{\varepsilon\to 0}  \Big| {1\over \varepsilon}\big[\vec{U}({t_0},x, {\cal L}_{\xi + \varepsilon \eta e_1}) -\vec{U}({t_0},x,{\cal L}_\xi)\big] - \delta  Y^{x, \xi, \eta e_1}_{{t_0}}\Big|^2\ = 0.
\end{eqnarray}
Thus, by the definition of $\partial_{\mu } \vec{U}$,
\begin{eqnarray}
	\label{pamuV1}
	\mathbb{E}\big[\partial_{\mu } \vec{U}({t_0},x, \mu, \xi) \eta e_1\big]=\mathbb{E}\big[\partial_{\mu_1 } \vec{U}({t_0},x, \mu, \xi) \eta\big] =  \delta  Y^{x, \xi, \eta e_1}_{{t_0}}.
\end{eqnarray}

{\it Step 2.} In this step we assume that $\xi$ (or say, $\mu$) is discrete: $p_i = \mathbb{P}(\xi = x_i)$, $i=1,\cdots, n$. {In particular, we have that $\mu=\sum_{i=1}^n p_i x_i$, for some $\{x_1,\dots,x_n\}\subset\R^d.$}
Fix $i$ and consider the following system of McKean--Vlasov FBSDEs: for $j=1,\cdots, n$,
\begin{eqnarray}
	\label{tdmuYij}
	\left\{\begin{array}{ll}
		\displaystyle  \nabla_{\mu_1} X^{i,j}_t&\displaystyle= 
		{\delta_{ij}}e_1 {-} \int_{{t_0}}^t \Big\{ \sum_{k=1}^n p_k \tilde {\mathbb{E}}_{{\cal F}_s}\Big[ (\nabla_{\mu_1} \tilde X^{i,k}_s)^{\top}\partial_{\mu p} H(X_s^{\xi,x_j},\rho_s, \tilde X^{\xi,x_k}_T,Y_s^{\xi,x_j}) \Big]\\[7pt]
		\displaystyle & \displaystyle+(\nabla_{\mu_1} X_s^{i,j})^{\top}\partial_{xp} H\big(X^{\xi,x_j}_s,\rho_s,Y^{\xi,x_j}_s)+(\nabla_{\mu_1} Y_s^{i,j} )^{\top}\partial_{pp} H\big(X^{\xi,x_j}_s,\rho_s,Y^{\xi,x_j}_s)\Big\}\dd s, \\[7pt]
		\displaystyle \nabla_{\mu_1} Y_t^{i,j}& = \displaystyle\partial_{xx} G(X_T^{\xi,x_j},\rho_T) \cdot\nabla_{\mu_1} X^{i,j}_T+\sum_{k=1}^n p_k \tilde {\mathbb{E}}_{{\cal F}_T}\Big[\partial_{\mu x} G(X_T^{\xi,x_j},\rho_T, \tilde X^{\xi,x_k}_T)\cdot \nabla_{\mu_1} \tilde X^{i,k}_T\Big] \\[7pt]
		\displaystyle &\displaystyle - \int_t^T  \Big\{\partial_{xx} H\big(X^{\xi,x_j}_s,\rho_s,Y^{\xi,x_j}_s)\cdot\nabla_{\mu_1} X^{i,j}_s+\partial_{px} H\big(X^{\xi,x_j}_s,\rho_s,Y^{\xi,x_j}_s)\cdot\nabla_{\mu_1} Y^{i,j}_s\\[7pt]
		\displaystyle&\displaystyle +\sum_{k=1}^n p_k \tilde {\mathbb{E}}_{{\cal F}_s}\Big[\partial_{\mu x} H(X_s^{\xi,x_j},\rho_s, \tilde X^{\xi,x_k}_s,Y_s^{\xi,x_j}) \cdot\nabla_{\mu_1} \tilde X^{i,k}_s \Big\} \dd s\\[7pt]
		\displaystyle& \displaystyle {-} \int_t^T\nabla_{\mu_1} Z_s^{0,i,j}\cdot \dd B_s^{0},
	\end{array}\right.
\end{eqnarray}
where $\delta_{ij}$ stands for Kronecker's symbol. In the above system $\nabla_{\mu_1} X^{i,j}_t$ represents perturbing $x_i$ in $\mu$ in the $e_1$ direction and measuring the variation in $X_t$ at $X^{\xi, x_j}$ (the place where $x_j$ has moved to by time $t$). The interpretation for $\nabla_{\mu_1} Y_t^{i,j}$ is similar. 

For any $\Phi\in\{X, Y, Z^0\}$, we define
\begin{eqnarray*}
	\nabla_{1} \Phi^{\xi, x_i} := \nabla_{\mu_1} \Phi^{i,i},\quad  \nabla_{1} \Phi^{\xi, x_i, *}:= {1\over p_i} \sum_{j\neq i} \nabla_{\mu_1} \Phi^{i,j} {\bf 1}_{\{\xi=x_j\}}.
\end{eqnarray*}
Note that $\Phi^\xi =  \sum_{j=1}^n \Phi^{\xi, x_j} {\bf 1}_{\{\xi=x_j\}}$. Since \eqref{tdmuYij} is linear, one can easily check that

\begin{align}
	\displaystyle  \nabla_1 X^{\xi, x_i}_t&\displaystyle=e_1{-} \int_{{t_0}}^t\Big\{(\nabla_1 X_s^{\xi,x_i})^{\top} \partial_{xp} H\big(X^{\xi, x_i}_s,\rho_s,Y^{\xi, x_i}_s)+(\nabla_1 Y_s^{\xi,x_i})^{\top} \partial_{pp} H\big(X^{\xi, x_i}_s,\rho_s,Y^{\xi, x_i}_s)
	\nonumber\\
	&\displaystyle +p_i\tilde {\mathbb{E}}_{{\cal F}_s}\Big[(\nabla_{1}\tilde X_s^{\xi,x_i})^{\top}\partial_{\mu p}H(X_s^{\xi,x_i},\rho_s,\tilde X_s^{\xi,x_i},Y_s^{\xi,x_i})\label{tdXxi-} \\
	& +(\nabla_1\tilde X_s^{\xi,x_i, *})^{\top}\partial_{\mu p}H(X_s^{\xi,x_i},\rho_s,\tilde X_s^{\xi},Y_s^{\xi,x_i}\Big]\Big\}\dd s, \nonumber
\end{align}
\begin{align} 
	\displaystyle  \nabla_1 X^{\xi, x_i, *}_t&\displaystyle=- \int_{{t_0}}^t\Big\{ (\nabla_1 X_s^{\xi,x_i, *})^{\top} \partial_{xp} H\big(X^{\xi}_s,\rho_s,Y^{\xi}_s)+(\nabla_1 Y_s^{\xi,x_i, *})^{\top}\partial_{pp} H\big(X^{\xi}_s,\rho_s,Y^{\xi}_s)\nonumber\\
	\displaystyle &\displaystyle +\tilde {\mathbb{E}}_{{\cal F}_s}\Big[(\nabla_{1}\tilde X_s^{\xi,x_i})^{\top}\partial_{\mu p}H(X_s^{\xi},\rho_s,\tilde X_s^{\xi,x_i},Y_s^{\xi})\label{tdXxi-new} \\
	& +(\nabla_1\tilde X_s^{\xi,x_i, *})^{\top}\partial_{\mu p}H(X_s^{\xi},\rho_s,\tilde X_s^{\xi},Y_s^{\xi})\Big]{\bf 1}_{\{\xi\neq x_i\}}\Big\}\dd s\nonumber
\end{align}
\begin{align} 
	\label{tdXxi-2}
	\displaystyle \nabla_1 Y_t^{\xi, x_i}&\displaystyle= \partial_{xx} G(X_T^{\xi,x_i},\rho_T) \cdot \nabla_1 X^{\xi, x_i}_T - \int_t^T\nabla_1 Z_s^{0,\xi,x_i}\dd B_s^{0}\nonumber\\
	\displaystyle&\displaystyle + p_i\tilde {\mathbb{E}}_{{\cal F}_T}\big[\partial_{\mu x} G(X_T^{\xi,x_i},\rho_T, \tilde X^{\xi,x_i}_T)\cdot \nabla_1\tilde X_s^{\xi,x_i}+\partial_{\mu x} G(X_T^{\xi,x_i},\rho_T, \tilde X^\xi_T)\cdot \nabla_1 \tilde X^{\xi, x_i, *}_T\big]\\
	\displaystyle &\displaystyle - \int_t^T \Big\{ \partial_{xx} H\big(X^{\xi,x_i}_s,\rho_s,Y^{\xi,x_i}_s)\cdot \nabla_1 X^{\xi,x_i}_s+ \partial_{px} H\big(X^{\xi,x_i}_s,\rho_s,Y^{\xi,x_i}_s)\cdot \nabla_1 Y^{\xi,x_i}_s\nonumber\\
	\displaystyle&\displaystyle + p_i\tilde {\mathbb{E}}_{{\cal F}_s}\big[\partial_{\mu x} H(X_s^{\xi,x_i},\rho_s, \tilde X^{\xi,x_i}_s,Y_s^{\xi,x_i})\cdot \nabla_1\tilde X_s^{\xi,x_i} 
	+ \partial_{\mu x} H(X_s^{\xi,x_i},\rho_s,  \tilde X^\xi_s,Y_s^{\xi,x_i})\cdot \nabla_1 \tilde X^{\xi, x_i, *}_T)\big]\Big\}\dd s \nonumber
\end{align}

\begin{align} 
	\displaystyle \nabla_1 Y_t^{\xi, x_i, *}&\displaystyle= \partial_{xx} G(X_T^{\xi},\rho_T)\cdot \nabla_1 X^{\xi, x_i, *}_T - \int_t^T\nabla_1 Y_s^{0,\xi,x_i, *}\cdot \dd B_s^{0}
	\nonumber\\
	\displaystyle&\displaystyle  +\tilde {\mathbb{E}}_{{\cal F}_T}\big[\partial_{\mu x} G(X_T^{\xi},\rho_T, \tilde X^{\xi,x_i}_T)\cdot \nabla_1\tilde X_s^{\xi,x_i} +\partial_{\mu x} G(X_T^{\xi},\rho_T, \tilde X^\xi_T)\cdot \nabla_1 \tilde X^{\xi, x_i, *}_T\big]{\bf 1}_{\{\xi\neq x_i\}}
	\nonumber\\
	\displaystyle &\displaystyle  - \int_t^T\bigg\{  \partial_{xx} H\big(X^{\xi}_s,\rho_s,Y^{\xi}_s)\cdot \nabla_1 X^{\xi,x_i, *}_s+ \partial_{px} H\big(X^{\xi}_s,\rho_s,Y^{\xi}_s)\cdot \nabla_1 Y^{\xi,x_i, *}_s\label{tdXxi-2new}\\
	\displaystyle&\displaystyle +\tilde {\mathbb{E}}_{{\cal F}_s}\big[\partial_{\mu x} H(X_s^{\xi},\rho_s, \tilde X^{\xi,x_i}_s,Y_s^{\xi})\cdot \nabla_1\tilde X_s^{\xi,x_i}
	\ +\partial_{\mu x} H(X_s^{\xi},\rho_s,  \tilde X^\xi_s,Y_s^{\xi})\cdot \nabla_1 \tilde X^{\xi, x_i-}_T)\big]{\bf 1}_{\{\xi\neq x_i\}}\bigg\}\dd s\nonumber
\end{align}

Since \eqref{tdmuFBSDE} is also linear, one can check that, for $\Phi\in\{ {X, Y, Z^0}\}$,
\begin{eqnarray}
	\label{tdXdecom}
	\delta  \Phi^{\xi, {\bf 1}_{\{\xi=x_i\}}e_1} = \nabla_1 \Phi^{\xi, x_i} {\bf 1}_{\{\xi=x_i\}}+ p_i \nabla_1   \Phi^{\xi, x_i, *}.
\end{eqnarray}
Moreover, note that
\begin{align*}
	\displaystyle & \tilde {\mathbb{E}}_{{\cal F}_T}\big[\partial_{\mu x} G(X_T^{x,\xi},\rho_T, \tilde X^\xi_T)\cdot\delta \tilde X^{\xi, {\bf 1}_{\{\xi=x_i\}}e_1}_T\big]\medskip\\
	\displaystyle &=   \tilde {\mathbb{E}}_{{\cal F}_T}\Big[\partial_{\mu x} G(X_T^{x,\xi},\rho_T, \tilde X^\xi_T)\cdot\big[ \nabla_1 \tilde X_T^{\xi, x_i} {\bf 1}_{\{\xi=x_i\}}+ p_i \nabla_1   \tilde X_T^{\xi, x_i, *}\big]\Big] \medskip\\
	\displaystyle&= p_i \tilde {\mathbb{E}}_{{\cal F}_T}\Big[\partial_{\mu x} G(X_T^{x,\xi},\rho_T, \tilde X^{\xi, x_i}_T)\cdot \nabla_1 \tilde X^{\xi, x_i}_T+\partial_{\mu x} G(X_T^{x,\xi},\rho_T, \tilde X^\xi_T)\cdot\nabla_1 \tilde X^{\xi, x_i, *}_T\Big]
\end{align*}
and similarly
\begin{align*}
	\displaystyle&  \tilde {\mathbb{E}}_{{\cal F}_s}\big[\partial_{\mu x} H(X_s^{x,\xi},\rho_s, \tilde X^\xi_s,Y_s^{x,\xi})\cdot\delta \tilde X^{\xi, {\bf 1}_{\{\xi=x_i\}}e_1}_s\big]\medskip\\
	\displaystyle&= p_i \tilde {\mathbb{E}}_{{\cal F}_s}\Big[\partial_{\mu x} H(X_s^{x,\xi},\rho_s, \tilde X^{\xi, x_i}_s,Y_s^{x,\xi})\cdot\nabla_1 \tilde X^{\xi, x_i}_s+\partial_{\mu x} H(X_s^{x,\xi},\rho_s,\tilde X^\xi_s,Y_s^{x,\xi})\cdot \nabla_1 \tilde X^{\xi, x_i, *}_s)\Big].
\end{align*}
Plug this into \eqref{eq:nabla mu X}, we obtain 
\begin{eqnarray}
	\label{tdYxipi}
	\delta \Phi_t^{x, \xi, {\bf 1}_{\{\xi=x_i\}}e_1}=  p_i \nabla_{\mu_1}\Phi^{x,\xi, x_i}_t,
\end{eqnarray}
where
\begin{eqnarray}
	\label{eq:nabla mu X2}
	\left.\begin{array}{ll}
		\displaystyle  \nabla_{\mu_1} Y^{x,\xi, x_i}_t&\displaystyle= \tilde {\mathbb{E}}_{{\cal F}_T}\big[\partial_{\mu x} G(X_T^x,\rho_T, \tilde X^{\xi,x_i}_T)\cdot  \nabla_1 \tilde X^{\xi, x_i}_T+\partial_{\mu x} G(X_T^x,\rho_T, \tilde X^\xi_T)\cdot  \nabla_1 \tilde X^{\xi, x_i, *}_T)\big]  \\[5pt]
		\displaystyle &\displaystyle {-} \int_t^T \bigg\{\partial_{px} H(X_s^{x,\xi},\rho_s,Y_s^{x,\xi})\cdot \delta Y^{x,\xi,x_i}_s\\[5pt]
		\displaystyle&\displaystyle+ \tilde {\mathbb{E}}_{{\cal F}_s}\Big[\partial_{\mu x} H(X_s^{x,\xi},\rho_s, \tilde X^{\xi, x_i}_s,Y_s^{x,\xi})\cdot\nabla_1 \tilde X^{\xi, x_i}_s+\partial_{\mu x} H(X_s^{x,\xi},\rho_s,\tilde X^\xi_s,Y_s^{x,\xi})\cdot \nabla_1 \tilde X^{\xi, x_i, *}_s\Big]\bigg\}\dd s\\[5pt]
		\displaystyle &\displaystyle{-}\int_t^T\nabla_{\mu_1} Z_s^{0,x,\xi, x_i}\cdot \dd B_s^{0}. 
	\end{array}\right.
\end{eqnarray}
In particular, by setting $\eta = {\bf 1}_{\{\xi=x_i\}}$ in \eqref{pamuV1} we obtain:
\begin{eqnarray}
	\label{pamuVdiscrete}
	\partial_{\mu_1 } \vec{U}(t_0, x, \mu, x_i) = \nabla_{\mu_1} Y^{x,\xi, x_i}_{t_0}.
\end{eqnarray}
We shall note that \eqref{tdXxi-}-\eqref{tdXxi-new}, \eqref{tdXxi-2}-\eqref{tdXxi-2new} is different from \eqref{tdYx} and \eqref{tdYx-}, so \eqref{pamuVdiscrete} provides an alternative discrete representation.

{\it Step 3.} We now prove \eqref{pamuV} in the case that $\mu$ is absolutely continuous. For each $n\ge 3$, set 
\[
x^n_{\vec i} := \frac{\vec i}{n}, \quad
\Delta_{\vec i}^n:=\left[\frac{i_1}{n},\frac{i_1+1}{n}\right)\times\cdots\times\left[\frac{i_d}{n},\frac{i_d+1}{n}\right),\quad \vec i=(i_1,\cdots,i_d)^{\top}\in\mathbb Z^d.
\]
For any $x\in\mathbb R^d$, there exists $\vec i(x):=(i_1(x),\cdots,i_d(x))\in \mathbb Z^d$ such that $x\in \Delta_{\vec i(x)}^n$. Let 
$$
\vec i^n(x):=(i^n_1(x),\cdots,i^n_d(x))\in \mathbb Z^d,\quad\mbox{where}\quad 
i^n_l(x):=\min\{\max\{i_l,-n^2\},n^2\}, ~ l=1,\cdots,d.
$$
Denote $Q_n:= \{x\in \mathbb{R}^d: |x_i|\le n, i=1,\cdots, d\}$, $\mathbb Z_n^d:=\{\vec i\in\mathbb Z^d\,:\,\Delta^n_{\vec{i}}\cap Q_{n}\not=\emptyset\}$, and 
\begin{eqnarray}
	\label{xin}
	\xi_n :=\sum_{\vec i\in \mathbb Z_n^d} x_{\vec i}^n{\bf 1}_{\Delta_{\vec i}^n}(\xi)+\frac{\vec i^n(\xi)}{n}{\bf 1}_{Q_n^c}(\xi). 
\end{eqnarray}
It is clear that $\lim_{n\to+\infty}\mathbb E\big[|\xi_n - \xi|^2\big]=0$ and thus $\lim_{n\to\infty} W_2({\cal L}_{\xi_n}, {\cal L}_\xi)=0$. Then for any  scalar random variable $\eta$, by stability of FBSDE \eqref{tdmuFBSDE} and  BSDE \eqref{eq:nabla mu X}, we derive from \eqref{pamuV1} that
\begin{eqnarray}
	\label{tdmuYconv}
	\mathbb{E}\Big[\partial_{\mu_1 } \vec{U}(t_0,x,\mu, \xi)\eta \Big] = \delta Y^{x,\xi, \eta e_1}_{t_0} = \lim_{n\to\infty} \delta Y^{x,\xi_n, \eta e_1}_{t_0}.
\end{eqnarray}

For each $\tilde x\in \mathbb{R}^d$, let $\vec i(\tilde x)$ be the $i$ such that $\tilde x \in \Delta_{\vec i}^n$, which holds when $n> |\tilde x|$. Then $\big({\cal L}_{\xi_n}, \frac{\vec i(\tilde x)}{n}\big) \to (\mu, \tilde x)$ as $n\to \infty$ in $W_2$ and as a sequence in $\R^d$, respectively. By the stability of FBSDEs \eqref{FBSDE1}-\eqref{FBSDE2}, we have as $n \to \infty$ that
$
X^{\xi_n, \frac{\vec{i}(\tilde x)}{n}}\to X^{\xi, \tilde x}
$ 
and
$
Y^{\xi_n, \frac{\vec{i}(\tilde x)}{n}} \to Y^{\xi, \tilde x},
$
as $n\to+\infty$, under the norm given by $\|A\| := \E \(\sup_t \abs{A_t}^2\)$. Moreover, since $\mu$ is absolutely continuous, 
$$
\mathbb{P}\Big(\xi_n = \frac{\vec{i}(\tilde x)}{n}\Big) = \mathbb{P}\Big(\xi \in \Delta_{\vec i}^n\Big) \to 0, \quad \text{as} \quad n\to \infty.
$$ 
Then by the stability of \eqref{tdXxi-}-\eqref{tdXxi-new}, \eqref{tdXxi-2}-\eqref{tdXxi-2new} and \eqref{eq:nabla mu X2} we can check that
\begin{equation}
	\label{xinconv}
	\lim_{n\to\infty}\bigg(\nabla_1 \Phi^{\xi_n,\frac{\vec{i}(\tilde x)}{n}},\; \nabla_1 \Phi^{\xi_n,\frac{\vec{i}(\tilde x)}{n}, *},\; \nabla_{\mu_1}\Phi^{x,\xi_n,\frac{\vec{i}(\tilde x)}{n}}\bigg) =\bigg(\nabla_1 \Phi^{\xi,\tilde x},\; \nabla_1\Phi^{\xi,\tilde x, *},\; \nabla_{\mu_1}\Phi^{x,\xi,\tilde x}\bigg).
\end{equation}
Now for any bounded function $\varphi\in C(\mathbb{R}^d)$, set $\eta=\varphi(\xi)$ in \eqref{tdmuYconv}, we derive from  \eqref{tdYxipi} that
\begin{eqnarray*}
	\mathbb{E}\Big[\partial_{\mu_1 } \vec{U}(t_0,x,\mu, \xi)\varphi(\xi)\Big] 
	=  \lim_{n\to\infty} \delta Y^{x,\xi_n, \varphi(\xi_n)e_1}_{t_0} 
	=  \lim_{n\to\infty} \sum_{\vec i\in \mathbb Z_n^d} \varphi\Big(x^n_{\vec i}\Big)\delta Y^{x,\xi_n, {\bf 1}_{\{\xi_n=x^n_{\vec i}\}}e_1}_{t_0}
\end{eqnarray*}
and so, 
\begin{eqnarray*}
	\mathbb{E}\Big[\partial_{\mu_1 } \vec{U}(t_0,x,\mu, \xi)\varphi(\xi)\Big]= \lim_{n\to\infty} \sum_{\vec i\in\mathbb Z^d_n} \varphi(x^n_{\vec i})\nabla_{\mu_1} Y^{x,\xi_n, x^n_{\vec i}}_{t_0} \mathbb{P}(\xi\in \Delta_{\vec i})  = \int_{\mathbb R^d} \varphi(\tilde x)\nabla_{\mu_1} Y^{x,\xi, \tilde x}_{t_0} \dd \mu(\tilde x).
\end{eqnarray*}
This implies \eqref{pamuV} immediately.

{\it Step 4.} We finally prove the general case. Denote $\psi(x,\mu, \tilde x):= \nabla_{\mu_1} Y^{x,\xi, \tilde x}_{t_0}$. By the stability of FBSDEs, $\psi$ is continuous in all the variables. Fix an arbitrary $(\mu, \xi)$. One can  construct $\xi_n$ such that  ${\cal L}_{\xi_n} $ is absolutely continuous and $\lim_{n\to\infty}\mathbb{E}[|\xi_n-\xi|^2]=0$. Then, for any $\eta= \varphi(\xi)$ as in Step 3, by \eqref{pamuV1} and Step 3 we have 
\[
\displaystyle \mathbb{E}\big[\partial_{\mu_1 } \vec{U}(t_0,x, \mu, \xi) \varphi(\xi)\big] =  \lim_{n\to \infty}  \delta Y^{x, \xi_n, \varphi(\xi_n)e_1}_{t_0}=
\lim_{n\to \infty}\mathbb{E}\big[ \psi(x, {\cal L}_{\xi_n}, \xi_n)\varphi(\xi_n)\big] = \mathbb{E}\big[ \psi(x, \mu, \xi)\varphi(\xi)\big], 
\]
which implies \eqref{pamuV} in the general case for $k=1$. Similarly, we can show \eqref{pamuV} for $k=1,\cdots, d$. 

Note that $\partial_{\mu_k }\vec{U}(t_0,x,\mu,\tilde x)$ is deterministic and so
\[
\abs{\partial_{\mu_k } \vec{U}(t_0,x,\mu,\tilde x)}^2 
= \E\left(\abs{{\partial_{\mu_k }}\vec{U}(t_0,x,\mu,\tilde x)}^2\right) 
= \E\left(\abs{\nabla_{\mu_k}Y_{t_0}^{x,\xi,\tilde x}}^2\right)
\] which is bounded by a universal constant by Corollary \ref{cor: muk Y bound}. 
\end{proof}

\begin{cor}\label{cor:short}
Assume that $G,H$ satisfy Assumptions \ref{assum: G} and \ref{assum: H} and $\delta$ is given in Lemma \ref{lem:delta}. Recall the system \eqref{FBSDE2} and define $\vec{U}(t_0,x,\mu):=Y_{t_0}^{x,\xi}$ for any $t_0$ with $T-t_0<\delta$. Then the following decoupled McKean--Vlasov FBSDE
\begin{equation}\label{FBSDE100}
\left\{
\begin{array}{lcl}
{X_t^{x}} &=& x  + {\beta}B_t^{{0,t_0}} \\
{Y_t^{x}} &=& \displaystyle G(X_T^{x}, \rho_T) {-}\int_t^T H(X_s^{x}, \rho_s, Y_s^{x,\xi}) \dd s - \int_t^T Z_s^{0,x} \dd B_s^{{t_0}}
\end{array}
\right.
\end{equation}
is well-posed on $[t_0,T]$ for any $x\in\R^d$. Define $V(t_0,x,\mu):=Y^{x}_{t_0}$. Then $V$ is the unique classical solution of the master equation \eqref{eq:master} and $\pa_xV=\vec{U}$ on $[t_0,T]\times \R^d\times\sP_2(\R^d)$.
\end{cor}
\begin{proof}
The proof is essentially the same as  \cite[Proposition 5.2]{MouZha22}. The only difference is that here we do not have the presence of idiosyncratic noise and we need to take care of the less regular data $G,H$. This can be done by using the smooth mollifier constructed in \cite[Section 3]{MouZha23}.
\end{proof}

\section{Long Time Well-Posedness for the Master Equation}\label{sec:5}

First we recall a short time existence result, \cite[Theorem 5.45]{Carmona2018}.

\begin{lem}\label{lem: short time}
Suppose that Assumptions \ref{assum: G} and \ref{assum: H} are satisfied. Then there exists a universal constant $c > 0$ and $V$ so that $V$ is a classical solution to the master equation on $[T - c, T]\times\R^d\times\sP_2(\R^d)$. Furthermore for each fixed $t \in [T - c, T]$, $V(t,\cdot,\cdot)$ satisfies the same assumptions as $G$ in Assumption \ref{assum: G}.  
\end{lem}

\begin{proof}
First note that Assumption \ref{assum: G} and \ref{assum: H} gives the regularity conditions for the short time existence that will not affect the size of the time interval (in the notation of \cite{Carmona2018} the terms bounded by $\Gamma$). 

We note that the length of the time interval, $c$, is a universal constant for us. Indeed we have that $\phi$ is the identity (since in the notation of \cite{Carmona2018}, $b = \alpha$ for us), $\lambda = \frac{1}{4c_0}$ (this is the strong convexity constant for the Lagrangian which was assumed for us in Assumption \ref{assum: H}), and $L$ is the sum of the Lipschitz constants of $\pa_x H$ and $\pa_x V$ in space and in measure with respect to $W_1$, which is bounded by a universal constant due to Corollaries \ref{cor: space and W2} and \ref{cor: muk Y bound}. 
\end{proof}

Finally we prove our global well-posedness result.

\begin{thm}\label{thm:main}
Suppose that $G,H$ are displacement monotone and satisfy Assumptions \ref{assum: G} and \ref{assum: H}.

Then there is a unique global in time classical solution $V$ to the master equation \eqref{eq:master} within the class of bounded $\pa_{xx}V$ and $\pa_{x\mu}V$. 
\end{thm}

\begin{proof}
We first prove uniqueness. Let $V, \ti V$ be two classical solutions to the master equation with bounded $\pa_{xx}V$, $\pa_{x\mu}V$, $\pa_{xx}\tilde V$, $\pa_{x\mu}\tilde V$. Because of the short time well-posedness of the system \eqref{FBSDE1} we obtain that $\pa_x V = \pa_x \ti V$. 
We will now use \eqref{eq:SPDE} to show that $V = \ti V$. Let $u, \ti u$ be the value functions given by \eqref{eqn: MFSrepform} and $\rho, \ti \rho$ be the associated solutions of the second equation of \eqref{eq:SPDE}. Since $\pa_x V = \pa_x \ti V$ we have that $\pa_x u = \pa_x \ti u$ and so from the second equation of \eqref{eq:SPDE} we obtain that $\rho = \ti \rho$. It now follows from Lemma \ref{lem: representationformula} that $V = \ti V$.

\indent Next we prove existence. Let $T>0$ be an arbitrary long time horizon.

We will repeatedly apply the short time existence result Lemma \ref{lem: short time}. 
By the results of the previous sections, any solution $V$ to the master equation will be such that $\pa_x V(t,\cdot,\cdot)$ is uniformly Lipschitz continuous in the $(x,\mu)$-variable (with respect to $W_1$ in the $\mu$-variable), and this constant depends only on $G, H$ and $T$.

We let $V_0$ be a short time solution to the master equation on $[T-\delta, T]$. For $k=1,2,\dots$, we can recursively define solutions $V_k$ to the master equation by letting $V_k$ be the short time solution on $[T - (k+1)\frac{\delta}2, T - k\frac{\delta}2]$ with the terminal condition $V_k(T - k\frac{\delta}2, \cdot) = V_{k-1}(T - k\frac{\delta}2, \cdot)$. Since $\delta$ is a universal constant (as in particular the Lipschitz constants of $\pa_x V_k(t,\cdot,\cdot)$ are universally bounded by a constant depending on $G, H$ and $T$) we will need only finitely many steps to cover the whole interval $[0,T]$. Because of the uniqueness proved above we have that the $V_k$'s agree where their domains overlap and so we can stitch these together to obtain a classical solution to the master equation. 
\end{proof}

\begin{rmk}
The stability of the classical solution to the master equation \eqref{eq:master} obtained above is expected. Indeed, this is the consequence of the notion of solution in the full $\mathcal{C}^{k}$ spaces and would follow along the same lines as the proof of \cite[Theorem 5.45]{Carmona2018} (which is based on an approximation argument). Based on the proof of that theorem, we expect the following result. Let $(H_{n},G_{n})_{n\in\mathbb{N}}$ be a sequence of data functions, satisfying our main assumptions, such that this sequence, and its derivatives up to order two, converge locally uniformly to $(H,G)$, and to its derivatives, respectively, as $n\to+\infty$. The convergence in the measure variable is with respect to $W_{1}.$ Suppose that $(H,G)$ satisfies our standing assumptions, including the displacement monotonicity. Then $(V_{n})_{n\in\mathbb{N}}$ (the solution to the master equation with data $(H_{n},G_{n})$) and its derivatives up to order two convergence locally uniformly to $V$ (the solution to the master equation with data $(H,G)$), and to its derivatives, respectively, as $n\to+\infty$. As this result is expected, we decided not to give further details on it.

It worth however mentioning that for given data $(H,G)$ satisfying suitable monotonicity conditions, it is a subtle question to construct smooth approximations which satisfy the same monotonicity conditions. In the case of Lasry--Lions monotone data, we refer to the discussion in \cite{MouZha23} on this matter.
\end{rmk}

{\bf Acknowledgements.} MB's work was supported by the National Science Foundation Graduate Research Fellowship under Grant No. DGE-1650604. MB and ARM acknowledge the support of the Heilbronn Institute for Mathematical Research and the UKRI/EPSRC Additional Funding Programme for Mathematical Sciences through the focused research grant ``The master equation in Mean Field Games''. ARM has also been partially supported by the EPSRC New Investigator Award ``Mean Field Games and Master equations'' under award no. EP/X020320/1 and by the King Abdullah University of Science and Technology Research Funding (KRF) under award no. ORA-2021-CRG10-4674.2. CM gratefully acknowledges the support by CityU Start-up Grant 7200684, Hong Kong RGC Grant ECS 21302521, Hong Kong RGC Grant GRF 11311422 and Hong Kong RGC Grant GRF 11303223. 

\bibliographystyle{alpha}
\bibliography{MeanField_revision}

\end{document}